\documentclass{article}

\usepackage{graphicx,url}
\usepackage{color}
\newcount\Comments  
\newcommand{\kibitz}[2]{\ifnum\Comments=0\textcolor{#1}{#2}\fi}

\usepackage{amsmath}
\usepackage{amssymb}
\usepackage[T1]{fontenc}
\usepackage{latexsym}
\usepackage{psfrag}
\usepackage{hyperref}
\usepackage{multirow}
\usepackage{tikz}
\usepackage{mathdots}
\usepackage{yhmath}
\usepackage{cancel}
\usepackage{color}
\usepackage{siunitx}
\usepackage{array,enumerate}
\usepackage{multirow}
\usepackage{amssymb}
\usepackage{gensymb}
\usepackage{tabularx}
\usepackage{booktabs}
\usepackage{dsfont,bbm}
\usetikzlibrary{fadings}
\usetikzlibrary{patterns}
\usetikzlibrary{shadows.blur}
\usetikzlibrary{shapes}

\DeclareMathAlphabet{\mathbx}{U}{BOONDOX-ds}{m}{n}

\newtheorem{theorem}{Theorem}
\newtheorem{prop}{Proposition}[section]
\newtheorem{definition}[prop]{Definition}
\newtheorem{remark}[prop]{Remark}
\newtheorem{example}[prop]{Example}

\newenvironment{proof}{\noindent{\textsc{Proof.}}} {$\hfill\Box$\vspace{0.1 cm}\\}

\newcommand{\R}{\mathbb R}

\newcommand{\N}{{\mathbb N}}

\newcommand{\C}{{\mathcal C}}

\newcommand{\I}{{\mathcal I}}
\newcommand{\PP}{\mathcal{P}}
\newcommand{\TT}{\mathcal{T}}

\newcommand{\MM}{\mathcal{M}}

\newcommand{\eps}{\varepsilon}
\newcommand{\Id}{\mathrm{Id}}
\newcommand{\Zero}{\mathbx{0}}
\newcommand{\FR}[1]{#1}
\newcommand{\FRc}[1]{}

\begin{document}

\title{Generalized solutions to bounded-confidence models}
\author{Benedetto Piccoli and Francesco Rossi
}                     
%
%
\maketitle
%


\begin{abstract}
Bounded-confidence models in social dynamics describe multi-agent systems, where each individual interacts only locally with others. Several models are written as systems of ordinary differential
equations with discontinuous right-hand side: this is a direct consequence
of restricting interactions to a bounded region with non-vanishing 
strength at the boundary. 
Various works in the literature analyzed properties of solutions,
such as barycenter invariance and clustering.
On the other side, the problem of giving a precise definition of solution,
from an analytical point of view, was often overlooked.
However, a rich literature proposing different concepts of solution to discontinuous differential equations is available. Using several concepts of solution, we show how existence is granted under general assumptions, while
uniqueness may fail even in dimension one, but holds for almost
every initial conditions.
Consequently, various properties of solutions depend on the used definition and initial conditions.
\end{abstract}

\section{Introduction}
In the last decades, researchers from many different fields explored the behavior of large systems of active particles or agents.
The latter, also called self-propelled, intelligent or greedy, refers
to entities with capability of decision making and, usually,
of altering the energy or other otherwise conserved quantities of the system. 
Examples include dynamics of opinions in social networks, animal groups,
networked robots, pedestrian dynamics and language evolution.
The dynamics is written as an Ordinary Differential Equation (ODE in the following) in large dimension and various mean-field, kinetic
and hydrodynamic limit descriptions were studied in the literature, see
\cite{ABFHKPPS19,Bertozzi,CCH,CFRT,CPT,DDM14,MBG16,PR18} and references therein.

One of the main phenomena is \emph{self-organization}
of the whole system, stemming from simple interaction rules at particle
level.  Such interaction rules are often motivated by relationships
among agents and thus referred to as \emph{social dynamics}
\cite{ACMPPRT17,PROSKURNIKOV201765,PROSKURNIKOV2018166}.
The most common self-organized configurations are: consensus  \cite{OFM07}, i.e. all agents reaching a common state; alignment, i.e. agents reach consensus on a subset of the state variables \cite{CFPT15}; clustering, i.e. agents grouping in a small number
of well-separated states \cite{JABIN20144165,MT14}.

Our attention is focused on \emph{bounded-confidence}
models, where each agent interact only with agents located
within a bounded surrounding zone
\cite{CFT12,haskovec2020simple,MB12}. 
One of the most well-known of such
model is the Hegselmann-Krause with agents interacting 
is placed within a given distance,
see \cite{blondelHK,HK}. 
A general model can be written as follows:
\begin{equation}\label{eq:HKintro}
\dot{x}_i=\sum_{j=1}^N a_{ij}(\|x_i-x_j\|) (x_j-x_i) \mbox{~~~~~with~~} a_{ij}(r)=\begin{cases}
\phi_{ij}(r) &\mbox{~~ if~} r\in[0,1)\\
0 &\mbox{~~ if~} r\in[1,+\infty)
\end{cases},
\end{equation}
where $x_i\in\R^n$ is the state of agent $i$ (e.g. position, opinion, speed), 
$N$ the number of agents. Functions
$\phi_{ij}:[0,1]\to\R^+$ represents the strength of interaction
between agent $i$ and $j$, that are supposed to be symmetric (i.e. $\phi_{ij}=\phi_{ji}$) from now on. 
The original model corresponds to  $\phi_{ij}\equiv 1$ and was written in discrete time. However, many extensions were considered in continuous time. 
As a consequence, we have the following crucial observation: the right hand side of \eqref{eq:HKintro} is a discontinuous function. For this reason, one needs to carefully select a concept of solution to such discontinuous ODE. In our opinion, such aspect has been often overlooked in the extensive literature about bounded-confidence models, with some notable exceptions, such as \cite{blondel1,blondel2,frasca}.

The study of ODEs with discontinuous right-hand side, dating back to Caratheodory, has played a crucial role in mathematical analysis and in control theory. We refer to \cite{clarkebook,Filippov,vinter} for an extensive overview of the subject. In this article, we will make use of the main concepts of solutions that have been defined in this context, and in particular we will discuss: {\bf classical, Caratheodory, Filippov, Krasovskii, Clarke-Ledyaev-Sontag-Subbotin (briefly CLSS), and stratified solutions}. We recall the precise definition of such solutions in Section \ref{s-sols} below.

It is easy to prove that classical solutions may not exist, but that they enjoy uniqueness. Instead, the first surprising result about solutions of the Hegselmann-Krause model will be the following.
\begin{theorem} \label{t-exun}
Consider \eqref{eq:HKintro} with $\phi_{ij}$ Lipschitz continuous
and $\phi_{ij}=\phi_{ji}$.
Then, there exists a solution (global in time) for every initial condition 
and for every definition of solution, except for classical.\\
Uniqueness of solutions does not hold for any of the definitions, except for classical (and for stratified for a fixed stratification). 
Nevertheless, uniqueness holds for almost every initial data for every definition.
\end{theorem}
The proof of the positive result can be found in Section \ref{s-main}.
Many examples, provided in the following sections, will show that the discontinuity can generate parameteric families of solutions.
The latter may be of combinatorial complexity in terms of the number
of agents $N$ and the dimension of the state space $n$.

After solving the questions about existence and uniqueness, we will focus on some properties of such solutions. In the rich literature about social dynamics models, some crucial properties of solutions were explored. Among them, we want to recall the following:
\begin{itemize}
\item[{\bf P1)}] The {\bf barycenter} $\bar{x}=\frac{1}{N}\sum_i x_i$ is {\bf invariant} along trajectories.
\item[{\bf P2)}] For every solution $x(\cdot)$, $x(t)$ converges for $t\to\infty$
to $x^{\infty}=(x^\infty_1,\ldots,x^\infty_N)\in\R^{nN}$, $x^\infty_i\in\R^n$, 
such that for every $1\leq i,j\leq N$ either $x^\infty_i=x^\infty_j$
or $\|x^\infty_i-x^\infty_j\|\geq 1$. This property is called {\bf clustering} and
the number of distinct agents among $x^\infty_i$ is the number of clusters.
\item[{\bf P3)}] The asymptotic state $x^\infty$ of P2)   only {\bf depends on the initial data} of the trajectory. In particular, the number of clusters only depends on the initial condition.
\end{itemize}
As we will see, each of such properties may fail to hold, depending on the concept of solution used. Indeed, our second main result is the following.
\begin{theorem}\label{t-prop}
Consider \eqref{eq:HKintro} with $\phi_{ij}$ Lipschitz continuous, $\phi_{ij}=\phi_{ji}$, then the following holds.\\
Classical solutions satisfy P1-2-3).\\
Caratheodory, Filippov, Krasovskii and CLSS solutions satisfy P1-P2) but not P3), in general.\\
Stratified solutions satisfy P1-2-3) for a fixed stratification,
but $x^\infty$ in P3) depends on the stratification.
\end{theorem}
The proof of Theorem \ref{t-prop} is given in Section \ref{s-main}.\\
It is remarkable to observe that both solutions and their properties drastically vary when replacing $a_{ij}$ in \eqref{eq:HKintro} even in a single point, e.g. by choosing $a_{ij}(1)=1$. Indeed, the following last main result holds.
\begin{theorem}\label{t-prop2}
Consider \eqref{eq:HKintro} with $\phi_{ij}$ Lipschitz continuous, $\phi_{ij}=\phi_{ji}$, and $a_{ij}(r)$ replaced by 
\begin{equation}\label{e:HK1}a_{ij}(r)=\begin{cases}
1 &\qquad\mbox{~~if~~~} r\in[0,1],\\
0 &\qquad\mbox{~~if~~~} r\in(1,+\infty).
\end{cases}
\end{equation}
The sets of Krasovskii and Filippov solutions coincide with the ones  of  \eqref{eq:HKintro}.\\
The sets of classical, Caratheodory, CLSS and stratified solutions are different in the two cases.\\
All statements of Theorems \ref{t-exun} and \ref{t-prop} hold true in this case too. 
\end{theorem}
This theorem shows that one cannot consider the right-hand side of an ODE as a $L^\infty$ function, since the structure of the solution actually depends on the chosen representative. The proof of Theorem \ref{t-prop2} is given in Section \ref{s-main}.\\

The structure of the article is the following. In Section \ref{sec:BD} we provide notations and definitions, including the various concepts of solution for
discontinuous ODEs. Section \ref{sec:HK} presents the generalized
Hegselmann-Krause model and various general properties,
while Section \ref{sec:HK-R} deals with the linear case in $\R$, already providing various examples of violating uniqueness and properties P1-2-3).
Section \ref{sec:HK-multiR} deals with the multidimensional case, providing other counterexamples. In Section \ref{sec:uniq}, we prove uniqueness
for almost every initial datum, while Section \ref{sec:P2} focuses on the clustering property P2). Finally, Section \ref{s-main} contains
the proofs of the main Theorems.

\section{Notations and definitions}\label{sec:BD}
In this article, we denote by  $\lambda^m$ the Lebesgue measure on $\R^m$. 
For $x\in\R^m$, $B(x,r)$ is the ball of radius $r>0$
centered at $x$ and $B(r)=B(0,r)$ is the ball centered at the origin.
A cone $K\subset \R^m$ is a set with $0\in K$ and such that
$\alpha\cdot K =\{ \alpha x:x\in K\}\subset K$
for every $\alpha>0$. Given an embedded manifold $M\subset\R^m$,
the symbol $\partial M$ denotes the topological boundary.
Given $A\subset \R^m$, we set
$$co(A)=\left\{\sum_{i=1}^\ell \alpha_i x_i : \ell\in\N, \lambda_i\in [0,1], \sum_i\lambda_i=1, x_i\in A \right\}$$
the convex hull of $A$, and denote by $\overline{co}(A)$ its closure.\\
We denote by $AC([0,T],\R^m)$ the space of absolutely continuous functions on a time interval $[0,T]$. Recall that every absolutely continuous function is
differentiable for almost every time, i.e. except for times on a set of zero Lebesgue measure.

We also introduce the following:
\begin{definition}\label{def:strat}
A set $\Gamma\subset\R^m$, $\Gamma=\cup_{i=1}^{m_\Gamma} M_i$, 
with $m_\Gamma\in\N\cup\{+\infty\}$ and $M_i$ being ${\cal C}^1$ embedded manifold of dimension $n_i\leq m$, is stratified if:
\begin{itemize}
\item[i)] The family $M_i$ is locally finite: given a compact $K$, it holds $K\cap M_i\not=\emptyset$ only for finite many $i$.
\item[ii)] for $i\not=j$  it holds $M_i\cap M_j=\emptyset$,  and if $M_i\cap \partial M_j\not=\emptyset$ then $M_i\subset \partial M_j$ and $n_i<n_j$.
\end{itemize}
We call $\max_i n_i$ the dimension of the stratified set $\Gamma$.
\end{definition}
\begin{remark}
For simplicity we used the definition of topological stratification, even
if the examples we consider will admit Whithney or Boltianskii-Brunovsky stratification.
We refer the reader to \cite{marigo_piccoli_2002,PS00,Suss90} for a discussion of the different concepts
and the role played for discontinuous ordinary differential equation and optimal feedback control.
\end{remark}

An autonomous Ordinary Differential Equation (briefly ODE) is written as:
\begin{equation}\label{eq:ODE}
\dot{x}(t)=f(x(t))
\end{equation}
where $x\in\R^m$ and $f:\R^m\to \R^m$ is a measurable and locally bounded
function (defined at every point). The different concepts of solution
will be discussed in the next Section \ref{s-sols}.

A multifunction on $\R^m$ is a function $V:\R^m\to \PP(\R^m)$,
with $\PP(\R^m)$ being the powerset of $\R^m$, i.e. the set of subsets of $\R^m$.
Given a multifunction $V$, one can consider the differential inclusion:
\begin{equation}\label{eq:diff-incl}
\dot{x}(t)\in V(x(t)).
\end{equation}
A solution is an absolutely continuous function $x(\cdot)$ which satisfies
\eqref{eq:diff-incl} for almost every $t$. 

We define the Hausdorff distance $d_H$ on the powerset of $\R^m$ 
as follows: given $x\in\R^m$ and $A,B\subset\R^m$ we set
$d(x,A)=\inf\{d(x,y):y\in A\}$ and 
$d_H(A,B)=\sup \{d(x,A),d(y,B):x\in B,y\in A\}$.
A multifunction $V$ is continuous if it is continuous for the Hausdorff
distance, while $V$ is upper semicontinuous at $x$ if for every $\epsilon>0$
there exists $\delta>0$ such that $V(y)\subset V(x)+B(\eps)$
for every $y$ with $|x-y|<\delta$.\\  
A continuous multifunction $V$ is also upper semicontinuous. It is well known that if $V$ is upper semicontinuous 
with compact convex values, then the corresponding differential inclusion \eqref{eq:diff-incl} admits solutions for
every initial condition, see \cite{AubinCellina}. More precisely, we
have the following:
\begin{prop}
Assume that the multifunction $V$ in \eqref{eq:diff-incl} is upper semicontinuous and, for every $x\in\R^m$, $V(x)$ is a nonempty, compact and convex subset of $\R^m$. Then for every initial condition $x_0$ there exists a solution to \eqref{eq:diff-incl}.\
Moreover, if $V$ satisfies $\sup_{v\in V(x)} |v|\leq C(1+\|x\|)$ for some $C>0$,
then for every $x_0\in\R^m$ and $T>0$, the set of solutions to \eqref{eq:diff-incl} with initial condition $x(0)=x_0$ is
a nonempty, compact, connected subset of $AC([0,T],\R^{m})$.
\end{prop}

\subsection{Solutions to discontinuous ordinary differential equations} \label{s-sols}
Given the ODE \eqref{eq:ODE} with $f$ discontinuous, it is convenient to define the associated Filippov multifunction as:
\begin{equation}\label{eq:Fil-multi}
F(x)=\bigcap_{\delta>0}\bigcap_{\lambda^m(N)=0}
\overline{co} \{f(y):y\in (x+B_\delta\setminus N)\}.
\end{equation}
We have the following proposition, see \cite{AubinCellina}.
\begin{prop}
Consider an ODE \eqref{eq:ODE} with $f$ measurable and locally bounded.
Then the corresponding Filippov multifunction $F$ defined by \eqref{eq:Fil-multi} is upper semicontinuous with nonempty,
compact and convex values, thus the differential inclusion $\dot{x}\in F(x)$
admits solutions for every initial condition. 
\end{prop}
Similarly, the Krasovskii multifunction,
associated to \eqref{eq:ODE},
is defined as:
\begin{equation}\label{eq:Kras-multi}
K(x)=\bigcap_{\delta>0}
\overline{co} \{f(y):y\in (x+B_\delta)\},
\end{equation}
and it shares the same regularity as the Filippov one; thus, solutions exist to the corresponding differential inclusion for every initial condition.

\begin{remark}
We will mostly consider examples of ODEs for which Filippov and Krasovsky multifuction coincide, so we will mainly focus on the Filippov definition. However, the set of solutions may differ significantly in the general case, as shown by Example \ref{ex:1} below.
\end{remark}

To define a third concept of solution, we introduce the following:
\begin{definition} \label{d-strat}
A stratification $S$ for the ODE \eqref{eq:ODE}
is a quadruplet $(\Gamma,N_1,N_2,\Sigma)$
with $\Gamma=\R^m$ stratified, $N_1\cup N_2=\{1,\ldots,m_\Gamma\}$,
$N_1\cap N_2=\emptyset$ and $\Sigma:N_2\to N_1$ such that 
the following holds:
\begin{itemize}
\item the manifolds $M_i$, $i\in N_1$, are called
type I cells and the manifolds $M_j$, $j\in N_2$, are called type II cells.
\item if $M_i$ is of type I, then $f(x)\in T_x M_i$ for every $x\in M_i$ and $f$ restricted to $M_i$ is smooth.
\item if $M_j$ is of type II, then  for every $x\in M_j$ there exist $\epsilon>0$ and a unique absolutely continuous curve  $\xi_x:[0,\epsilon[\to \R^m$ with  $\xi_x(0)=x$, $\xi_x(t)\in M_{\Sigma(j)}$ for $t\in ]0,\epsilon[$ and
$\dot{\xi}_x(t)=f(\xi_x(t))$ for every $t\in ]0,\epsilon[$.
\end{itemize}
\end{definition}

Many definitions of solutions for \eqref{eq:ODE} are then available, most of which coincide when $f$ is sufficiently regular (e.g. locally Lipschitz). 
We summarize in the following definition the concepts we are considering in the rest of the paper.
\begin{definition}\label{def:ODE-sol}
Given the ODE \eqref{eq:ODE} and $T>0$ we define the following:
\begin{enumerate}
\item A {\bf classical solution} is a function $x:[0,T]\to \R^m$, which is differentiable
and satisfies \eqref{eq:ODE} at every time $t\in [0,T]$
(with one-sided derivatives at $0$ and $T$).
\item A {\bf Caratheodory solution} is an absolutely continuous function $x:[0,T]\to \R^m$  which satisfies \eqref{eq:ODE} at almost every time $t\in [0,T]$.
\item A {\bf Filippov solution} is an absolutely continuous function $x:[0,T]\to \R^m$, which satisfies: 
\[
\dot{x}\in F(x(t))
\]
for almost every time $t\in [0,T]$, with $F$ given by \eqref{eq:Fil-multi}.
\item A {\bf Krasovskii solution} is  is an absolutely continuous function $x:[0,T]\to \R^m$, which satisfies: 
\[
\dot{x}\in K(x(t))
\]
for almost every time $t\in [0,T]$, with $K$ given by \eqref{eq:Kras-multi}.
\item A {\bf limit of sample-and-hold solution or Clarke-Ledyaev-Sontag-Subbotin (briefly CLSS) solution} is a continuous function $x:[0,T]\to \R^m$, 
which is uniform limit of continuous and
piecewise smooth functions $x_\nu$, $\nu\in\N$, 
for which there exist $0=t^{0}_{\nu}<t^{1}_{\nu}<\cdots <t^{m_\nu}_{\nu}=T$
such that $\dot{x}_\nu(t)=f(x_\nu (t^{j}_{\nu}))$ for $t\in [t^{j}_{\nu},t^{j+1}_{\nu}[$,
$j=0,\ldots, m_\nu-1$,
and $\max_{j} (t^{j+1}_{\nu}-t^{j}_{\nu})\to 0$ as $\nu\to\infty$.
\item If $S=(\Gamma,N_1,N_2,\Sigma)$ is a stratification for $f$, then {\bf a stratified solution} generated by $S$
is a continuous and piecewise smooth function $x:[0,T]\to \R^m$ for which there exist $0=t_0 < t_1< t_2<\cdots <t_\ell=T$
and $i_1,\ldots,i_{\ell}\in \{1,\ldots , m_{\Gamma}\}$ such that the following holds for $k=0,\ldots,\ell-1$: if $i_k\in N_1$ then $x(\cdot)$ is a classical solution on $[t_k,t_{k+1}[$ contained in $M_{i_k}$, while if $i_k\in N_2$ then $x(t_k)\in M_{i_k}$ and $x(\cdot)$  is a classical solution 
on $]t_k,t_{k+1}[$ contained in $M_{\Sigma(i_k)}$.
\item A solution $x:[0,T]\to \R^m$ (in one of the previous senses) is said robust if there exists a neighborhood $N$ of $x(0)$ and, for every $y\in N$,
a solution $x_y$ with $x_y(0)=y$ such that the following holds: for each $y_\nu\in N$, with $y_\nu\to x(0)$ as $\nu\to +\infty$,
$x_{y_\nu}$ converges to $x$ uniformly on $[0,T]$.
\item A solution $x:[0,T]\to \R^m$ (in one of the previous senses) is said 
cone-robust if there exists a cone $K$ with nonempty interior, 
a neighborhood $N$ of $x(0)$ and,  for every $y\in ((x+K)\cap N)$,
a solution $x_y$ with $x_y(0)=y$ such that the following holds: for each $y_\nu\in (x+K)\cap N$, with $y_\nu\to x(0)$ as $\nu\to +\infty$,
$x_{y_\nu}$ converges to $x$ uniformly on $[0,T]$.
\end{enumerate}
\end{definition}

\begin{remark} \label{rem-cara}
The concept of classical solution is not used for discontinuous ODEs, because of general lack of existence. Instead, Caratheodory solutions are the one commonly used, as they are equivalent to solutions in the integral form:
\[
x(t)=x(0)+\int_0^t f(x(s))\,ds.
\]
The concepts of Filippov and Krasovskii solutions are commonly used
to deal with general discontinuous ordinary differential equations. They have the advantage of being based on the well-developed
theory of differential inclusions, see \cite{AubinCellina,Filippov}.\\
CLSS solutions have been introduced to provide a suitable concept
for discontinuous stabilizing feedbacks \cite{CLSS}. Notice that the sample-and-hold approximations are indeed numerical solutions provided by the explicit Euler scheme. Thus CLSS solutions represent solutions which may be generated by a numerical scheme in the theoretical limit.\\
The concept of stratification and stratified solution is particularly convenient
in optimal control theory, especially to build optimal synthesis,
see \cite{PS00}.
The concept of robust and cone-robust are useful to isolate
solutions in the same context \cite{marigo_piccoli_2002}.
\end{remark}

\begin{figure}[h]
\begin{center}
\includegraphics[scale=.3]{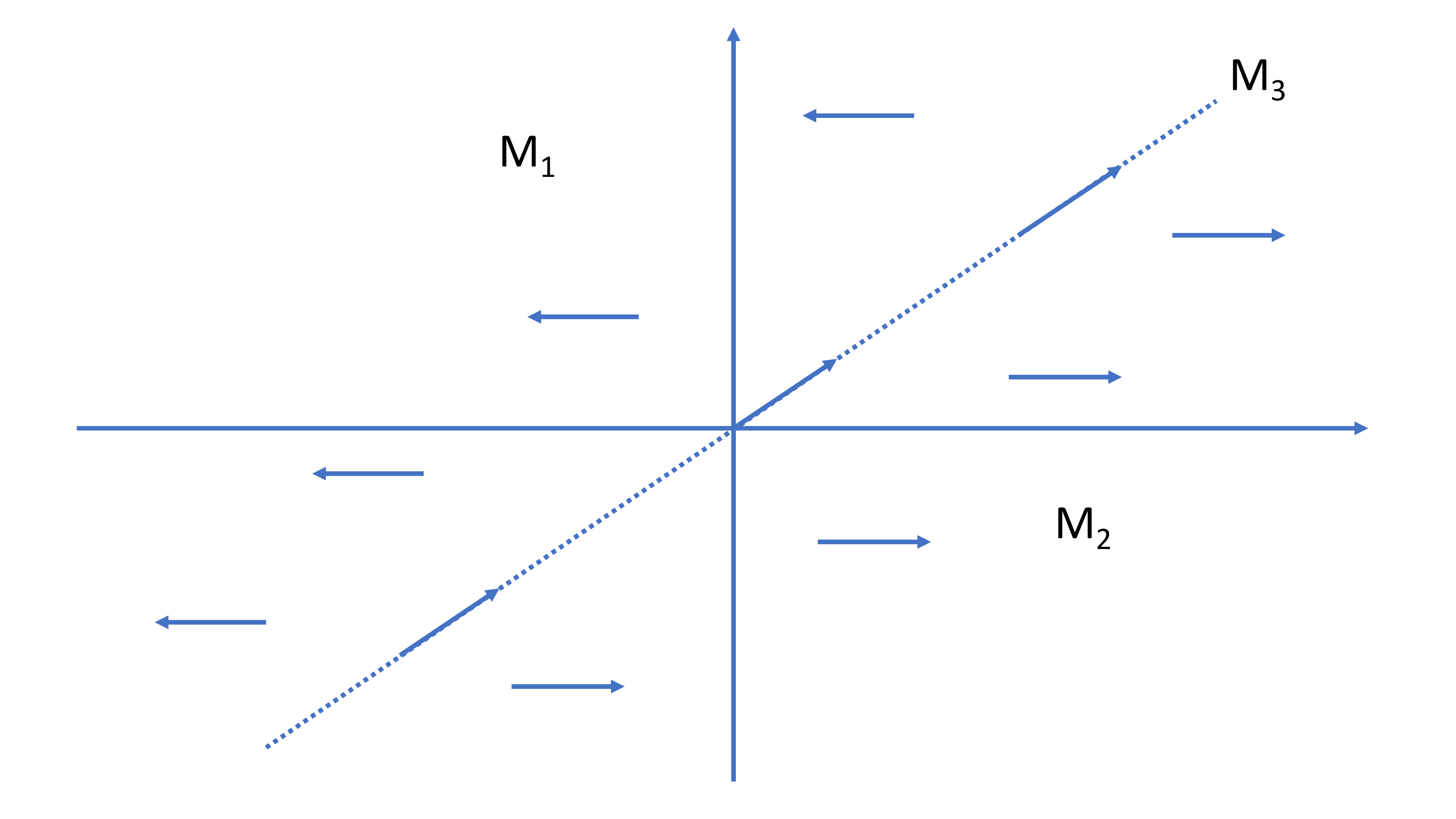}
\caption{Graphical representation of $f$ and stratifications
for Example \ref{ex:1}.}
\label{fig:example1}
\end{center}
\end{figure}

The different concepts give rise to very different sets of solutions,
as illustrated by next Example.
\begin{example}\label{ex:1}
Consider the ODE \eqref{eq:ODE} on $\R^2$ with initial condition 
$x(0)=(x_1(0),x_2(0))$ and $f$ given by:
\begin{equation}\label{eq:example-R2}
f(x)=\left\{
\begin{array}{ll}
(-1,0) & x_2>x_1\\
(1,1) & x_2=x_1\\
(1,0) & x_2<x_1
\end{array}
\right.
\end{equation}
See also Figure \ref{fig:example1} for a graphical representation of $f$.
Clearly for $x_2(0)\not= x_1(0)$ 
there exists a unique solution given by $x_2(t)=x_2(0)$
and $x_1(t)=x_1(0)\pm t$ if $x_1(0) \gtrless x_2(0)$ (for all concepts).
Therefore, we focus on the initial condition for which $x_1(0)=x_2(0)$ and
without loss of generality we assume $x_1(0)=x_2(0)=0$.
We notice that $F((0,0))=\{(\alpha,0):\alpha\in [-1,1]\}$ for 
$F$ defined by \eqref{eq:Fil-multi} and 
$K(0,0)$ is the convex hull of the three points $(-1,0)$, $(1,1)$
and $(1,0)$ for $K$ defined by \eqref{eq:Kras-multi}.
Then, the set of solutions is as follows:
\begin{itemize}
\item There exists a unique classical solution given by
$x(t)=(t,t)$.
\item There exist two one-parameter families of Caratheodory solutions: 
for fixed $\bar{t}\in [0,+\infty]$ consider the continuous function $x^\pm$ such that
$x^\pm(t)=(t,t)$ on $[0,\bar{t}[$
and $x^\pm(t)=(\bar{t}\pm (t-\bar{t}),\bar{t})$ on $]\bar{t},+\infty[$.
\item
The set of Filippov solutions is given by 
two one-parameter families:
for fixed $\bar{t}\in [0,+\infty]$, consider the continuous function $x^\pm$ such that
$x^\pm(t)=(0,0)$ on $[0,\bar{t}[$
and $x^\pm(t)=(\pm (t-\bar{t}),0)$ on $]\bar{t},+\infty[$.
\item The set of Krasovskii solutions includes Caratheodory and Filippov solutions, and is given by the following infinite dimensional family. 
Given $\bar{t}\in [0,+\infty]$ and a Lipschitz continuous function
$\varphi:[0,+\infty[\to\R$ with $0\leq \varphi'(t)\leq 1$ for almost every $t$,
define the continuous function $x^{\bar{t},\pm}_\varphi$ such that
$x^{\bar{t},\pm}_\varphi(t)=(\varphi(t),\varphi(t))$ on $[0,\bar{t}[$ and
$x^{\bar{t},\pm}_\varphi(t)=(\varphi(\bar{t})\pm (t-\bar{t}),\varphi(\bar{t}))$ on $[\bar{t},+\infty[$. 
\item The only CLSS solution coincides with the classical one.
\item There exists three possible stratifications: $S_i$, $i=1,2,3$ defined as follows. First set $M_1=\{(x_1,x_2):x_2>x_1\}$, $M_2=\{(x_1,x_2):x_2<x_1\}$,
$M_3=\{(x_1,x_2):x_2=x_1\}$, $\Gamma=\R^2=\cup_i M_i$ and $m_\Gamma=3$.\\
The first stratification is $S_1=\{\Gamma,\{1,2\},\{3\}, \Sigma_1\}$, with $\Sigma_1(3)=1$.
The only stratified solution for $S_1$ is $x(t)=(-t,0)$.\\
The second is $S_2=\{\Gamma,\{1,2\},\{3\}, \Sigma_2\}$, with $\Sigma_2(3)=2$.
The only stratified solution for $S_2$ is $x(t)=(t,0)$.\\
Finally, the third is $S_3=\{\Gamma,\{1,2,3\},\emptyset,\emptyset\}$
and the only stratified solution for $S_3$ is $x(t)=(t,t)$.
\item No solution is robust and the only cone robust are
the stratified solutions for $S_1$ and $S_2$.
\end{itemize} 
\end{example}

\section{The Hegselmann-Krause model}\label{sec:HK}
One of the most known examples of \emph{social dynamics} is the
celebrated Hegselmann-Krause (briefly HK) model:
\begin{equation}\label{eq:HK}
\dot{x}_i=\sum_{j=1}^N a_{ij}(\|x_i-x_j\|) (x_j-x_i) \mbox{~~~~~with~~} a_{ij}(r)=\begin{cases}
\phi_{ij}(r) &\mbox{~~ if~} r\in[0,1)\\
0 &\mbox{~~ if~} r\in[1,+\infty).
\end{cases}
\end{equation}
where $x_i\in\R^n$, $i=1,\ldots,N$,
$\phi_{ij}:[0,1]\to\R^+$ are Lipschitz continuous, and $\phi_{ij}=\phi_{ji}$.
Each $x_i$ represents the (possibly multidimensional) opinion of the $i$-th agent.
To be precise, the original model was formulated in discrete-time with $\phi_{ij}\equiv 1$, see \cite{HK}. Obviously, existence and uniqueness for the discrete time version is granted, while \eqref{eq:HK} is the natural extension to the continuous-time case with $\phi_{ij} $ arbitrary.
We will provide various examples of lack
of uniqueness and some positive results.
Many examples will be provided for the special case $\phi_{ij}\equiv 1$,
i.e. with linear dynamics for interacting agents, while results
will be given for the general case. 

\subsection{Relationships between concepts of solution}

In this section, we will prove first results about the connection between different kinds of solutions. In particular, we will prove the following result.
\begin{prop} \label{p-inclusion}
The set of Filippov solutions to \eqref{eq:HK} coincides with the set of Krasovskii solutions and contains the set of Caratheodory solutions.
The set of Caratheodory solutions includes classical, CLSS and stratified solutions.
\end{prop}
The system \eqref{eq:HK} can be written in standard from 
\eqref{eq:ODE} by setting $m=nN$,
$x=(x_1,\ldots,x_N)\in\R^{nN}$,
$f=(f_1,\ldots,f_N)$ with $f_i:\R^n\to\R^n$ given
by \eqref{eq:HK}.
To prove some general properties of the system \eqref{eq:HK},
we first need to provide some definition.
\begin{definition}\label{def:M-disc}
Given $i,j\in\{1,\ldots,N\}$, $i\not= j$, we define
the subset of $\R^{nN}$:
\begin{equation}
\MM_{ij}=\{(x_1,\ldots,x_N): \|x_i-x_j\|=1\},\quad
\end{equation}
and the union of such subsets as:
\begin{equation} 
\MM=\cup_{i,j:i\not= j}\MM_{ij}.
\end{equation}
For $x\in\MM$ we let $J(x)=\{J_1,\ldots,J_{\ell(x)}\}$ 
be the unique partition of  $\{1,\ldots,N\}$ 
(i.e. $J_k\subset \{1,\ldots,N\}$ are disjoint
and $\cup_k J_k=\{1,\ldots,N\}$)
such that both $j_1\in J_k$ and $j_2\in J_k$ for some $k$
if and only if $x_{j_1}=x_{j_2}$.
\end{definition}
We have:
\begin{prop}\label{prop:M}
The map $f=(f_1,\ldots,f_N)$, with $f_i:\R^n\to\R^n$ given
by \eqref{eq:HK}, is locally Lipschitz continuous
at every  $x\in\R^{nN}\setminus \MM$.
Moreover, the set $\MM$ is stratified.
\end{prop}
\begin{proof}
The locally Lipschitz continuity of $f$ outside $\MM$ follows directly from the definition of $f_i$.\\
The set $\MM$ is stratified by defining the strata
as follows. Given any partition $J=\{J_1,\ldots,J_{\ell}\}$ 
of  $\{1,\ldots,N\}$, we set 
$M_J=\{x:J(x)=J\}$. Notice that $dim(M_J)=\ell$.
Property i) of Definition
\ref{def:strat} follows from the finiteness of partitions
of  $\{1,\ldots,N\}$. For property ii), 
write $J^1\prec J^2$ if the partition $J^1$ is a strict refinement of $J^2$. Then it is easy to check that $J^1\prec J^2$
if and only if $M_{J_1}\subset\partial M_{J_2}$ and,
in this case, $dim(M_{J_1})<dim(M_{J_2})$.
\end{proof}

\begin{prop}\label{prop:FK-HK}
Let $F$ be the Filippov multifunction defined as in \eqref{eq:Fil-multi}
for $f=(f_1,\ldots,f_N)$, with $f_i$ given by the right hand side of \eqref{eq:HK}. It holds $F(x_1,\ldots,x_N)=(F_1,F_2,\ldots,F_N),$ where
\begin{equation}\label{e:HKF}
F_i=\left\{\sum_{j\neq i : \|x_i-x_j\|=1} \alpha_j \phi_{ij}(1) (x_j-x_i)\ :\ \alpha_j\in[0,1]\right\}+ \sum_{j\neq i : \|x_i-x_j\|<1} \phi_{ij}(\|x_i-x_j\|) (x_j-x_i).\end{equation}
There exists $C>0$ such that  $\sup_{v\in F(x)} |v|\leq C(1+\|x\|)$, thus for every $x_0\in\R^{nN}$ and $T>0$, the set of Filippov solutions to \eqref{eq:HK} with initial condition $x(0)=x_0$ is a nonempty, compact, connected subset of $AC([0,T],\R^{m})$.\\
Moreover, the Krasovskii multifunction $K$ defined as in \eqref{eq:Kras-multi} coincides
with that defined by \eqref{eq:Fil-multi}, thus the set of Krasovskii solutions
coincide with the set of Fillippov solutions.\\
Finally, the property P1) holds for Filippov and Krasovskii solutions.
\end{prop}
\begin{proof} The explicit expression \eqref{e:HKF} can be verified by computation. Set $C_{ij}=\sup_{r\in [0,1]}\phi_{ij}(r)<+\infty$,
$i,j=1,\ldots,N$,
and $C'=\max_{ij}C_{ij}$.
Given $x\in\R^{nN}$, we have 
$\|x_i-x_j\|\leq \|x_i\|+\|x_j\| \leq \sqrt{2}\|x\|$, thus 
$\|f_i(x)\|\leq N\sqrt{2}C'\|x\|$.
Finally, $\|f(x)\|\leq N\sqrt{2N}C'\|x\|$
and the sublinear estimate holds for $F$.\\
Now fix $x\in\MM$ and consider 
$J(x)=\{J_1,\ldots,J_{\ell(x)}\}$.
Define the open set
\[
A(x)=\left\{y: \forall k\in\{1,\ldots,\ell(x)\},\ \forall\ i,j\in J_k,
\ i\not= j,\
\text{we have}\ \|y_i-y_j\|>1\right\}.
\]
Then $f(x)=\lim_{y\to x,y\in A(x)} f(y)$. 
Thus, the definition of $K$ given by \eqref{eq:Kras-multi} coincides with $F$ given by \eqref{eq:Fil-multi}.\\
Last statement was proved in \cite{frasca} for the case $n=1$, and can be easily adapted to the case $n\geq 1$. Indeed, observe that any vector field 
$(v_1,\ldots,v_N)\in F(x)$ satisfies $\sum_{i=1}^N v_i=0$, thus any (convex) combination of vector fields in $F(x)$ satisfies it too. The barycenter 
$\bar x$ is a continuous function satisfying $\dot{\bar{x}}=\sum_{i=1}^N v_i=0$ for a.e. time, then it is constant.
\end{proof}

\begin{prop}\label{prop:HK-Car}
Consider the system \eqref{eq:HK} with $\phi_{ij}$ Lipschitz continuous.
Then, the set of Caratheodory solutions is contained within the set of Filippov and Krasovskii solutions.
\end{prop}
\begin{proof}
Notice that $f(x)$ is continuous outside $\MM$
and, as in the proof of Proposition \ref{prop:FK-HK}, it holds
$f(x)=\lim_{y\to x,y\in A(x)}f(y)$. Therefore
$f(x)\in F(x)$ for all $x\in\R^{nN}$, thus we conclude.\\
Since a solution $x(\cdot)$ in Caratheodory sense
satisfies the equation for almost every time,
one has $\dot{\bar{x}}(t)=0$ for almost every $t$, thus P1) holds true.
\end{proof}

We are now ready to prove the inclusions given in Proposition \ref{p-inclusion} above.

{\noindent{\textsc{Proof of Proposition \ref{p-inclusion}.}} We proved in Proposition \ref{prop:FK-HK} that Filippov and Krasovskii solutions coincide. We proved in Proposition \ref{prop:HK-Car} that Caratheodory solutions are included in the set of Filippov solutions. By definition, stratified solutions
are Krasovskii solutions and also  satisfy the equation for almost every time,
thus they are also Caratheodory solutions.
Since both CLSS and Caratheodory solutions are Lipschitz functions of time (due to boundedness of the right hand side), one has that CLSS solutions are Caratheodory: indeed, they can be seen as limits of Euler explicit schemes for Caratheodory solutions.
Finally, it is also evident from the Definition \ref{def:ODE-sol} (and Remark \ref{rem-cara}) that classical solutions are also Caratheodory ones. {$\hfill\Box$\vspace{0.1 cm}\\}

\subsection{Existence of solutions} \label{s-existence}
We now deal with existence of solutions. The existence of Fillippov (and Krasovskii) solutions are guaranteed
by the general theory of differential inclusions, as recalled in Proposition \ref{prop:FK-HK}. Also, CLSS solutions exist, as they are uniform limits of Lipschitz approximated trajectories.  We now prove that, for every initial datum there exists at least one Caratheodory solution defined for all times under the more general conditions of $\phi_{ij}$ only continuous.
\begin{prop}
Let us consider the general HK system \eqref{eq:HK} and assume
that $\phi_{ij}\in \C([0,1],]0,+\infty[)$. Then for every initial datum 
$\bar{x}\in \R^{nN}$ there exists at least one Caratheodory solution defined 
for all times $t\geq 0$.
\end{prop}
\begin{proof}
Given an initial condition $\bar{x}$, if $\bar{x}\in\R^{nN}\setminus\MM$
then \eqref{eq:HK} is continuous and locally bounded, thus by Peano
Theorem there exists a local (in time) solution $x(\cdot)$. 
This solution can be extended until the first time $t^1_\MM$
such that $x^1=x(t^1_{\MM})\in\MM$.\\
Given a (not directed) graph $G=(V,E)$, with $V=\{v_1,\ldots, v_n\}$,  we consider the equation
\begin{equation}\label{eq:G-sys}
\dot{x_i}=\sum_{j:\{i,j\}\in E} a_{ij}(\|x_i-x_j\|) (x_j-x_i).
\end{equation}
Since \eqref{eq:G-sys} has a continuous right-hand side, by Peano Theorem
for a fixed $G$ and an initial condition, there always exist a solution to \eqref{eq:G-sys}. Thus our strategy is to construct a $G$ so that
the solution to \eqref{eq:G-sys} from $x^1$ is also a 
Caratheodory solution to \eqref{eq:HK}.\\
Define $\I=\{\{i,j\}: \|x^1_i-x^1_j\|=1\}$ and let $G_1=(V_1,E_1)$
be the (not directed) graph with $V_1=\{v_1,\ldots, v_n\}$ and 
$\{i,j\}\in E_1$ if and only if $|x^1_i-x^1_j|< 1$. We now build a new graph 
$G_1'=(V_1,E_1')$, with $E_1\subset E_1'$.
First we order the elements of $\I$, then we proceed as follows
by recursion on the elements of $\I$.
If $\{i,j\}\in\I$ then set:
\[
\alpha_{ij}=(x^1_i-x^1_j)\cdot
\left( \sum_{(i,k)\in E_1} a_{ik} (x^1_k-x^1_i)- 
\sum_{(j,k)\in E_1} a_{jk} (x^1_k-x^1_j)\right),
\]
\[
\alpha_{ij}'=(x^1_i-x^1_j)\cdot
\left( \sum_{(i,k)\in E_1\cup \{\{i,j\}\}} a_{ik} (x^1_k-x^1_i)- 
\sum_{(j,k)\in E_1\cup \{\{i,j\}\}} a_{jk} (x^1_k-x^1_j)\right),
\]
where, for simplicity, we dropped the arguments in $a_{ik}$ and $a_{jk}$.
We add the edge $\{i,j\}$ to $E_1$ if and only if $\alpha_{ij}\leq \phi_{ij}(1)$.
Now, if $\alpha_{ij}>\phi_{ij}(1)>0$ then $\|x_i-x_j\|$ is increasing
along the solution to \eqref{eq:G-sys} for $G=G_1$.
Otherwise, since $\alpha_{ij}'=\alpha_{ij}-2 \phi_{ij}(1)$ and
$\alpha_{ij}\leq \phi_{ij}(1)$, then $\alpha_{ij}'<-\phi_{ij}(1)<0$, 
thus $\|x_i-x_j\|$ is decreasing along the solution to \eqref{eq:G-sys} 
for $G$ obtained from $G_1$ by adding the edge $\{i,j\}$.
In both cases the dynamics given by the graph is compatible with 
\eqref{eq:HK}.\\
Let $G_1'$ be the graph obtained at the end. We have that the
solution to \eqref{eq:G-sys} for $G=G_1'$ is also a Caratheodory
solution to \eqref{eq:HK}
on some interval $[t^1_{\MM},t^1_{\MM}+\delta_1]$ with $\delta_1>0$.\\
Let now $T>0$ be the maximal time so that $x(\cdot)$ can be defined
on $[0,T]$.  Assume, by contradiction, $T<+\infty$.
Then by the boundedness of $\phi_{ij}$,
we have that $x(\cdot)$ is Lipschitz continuous, thus we can define
$x(T)$. Applying the same reasoning as for $x(t^1_{\MM})$
we can extend the solution beyond $T$, thus reaching a contradiction.
\end{proof}

As for stratified solutions, their existence is ensured by Definition \ref{d-strat} itself as proved in next Proposition.
\begin{prop}
Let us consider the general HK system \eqref{eq:HK} with
$\phi_{ij}\in \C([0,1],]0,+\infty[)$. Then for every stratification
and  initial datum 
$\bar{x}\in \R^{nN}$, there exists a unique stratified solution defined 
for all times $t\geq 0$.
\end{prop}
\begin{proof}
Let $M_{i_0}$ be the stratum so that $\bar{x}\in M_{i_0}$. If $M_{i_0}$ is of type I then a local solution $x(\cdot)$ exists since $f$ is smooth on $M_{i_0}$. 
Let $t_1=\sup \{t: x(t)\in  M_{i_0} \}$, then by boundedness of $\phi_{ij}$
there exists $x_1=\lim_{t\nearrow t_1} x(t)$.\\
If $M_{i_0}$ is of type II, then by definition there exists a local solution 
$\xi_{\bar{x}}$ belonging to $M_{\Sigma(i_0)}$ for positive times. In this case we define
$t_1=\sup \{t: x(t)\in  M_{\Sigma(i_0)} \}$ and, by boundedness of $\phi_{ij}$s,
there exists $x_1=\lim_{t\nearrow t_1} x(t)$.\\
In both cases we let $M_{i_1}$ be the stratum so that $x_1\in M_{i_1}$
and proceed by recursion.\\
Again by  boundedness of $\phi_{ij}$, we can prolong the solution for every time. Moreover, such solution is unique by definition of stratification for
\eqref{eq:HK}.
\end{proof}

\subsection{Contractivity of the support}

In this section, we prove that the support of solutions (in any of the sense given above) is weakly contractive. This is a well-known property of Caratheodory solutions of HK models, see e.g. \cite{blondel2}. The proof of such property for Krasovskii solutions on the real line can be found in \cite[Proposition 3.iii]{frasca}. We will give a general proof for Krasovskii solutions 
in any dimension, again in the more general case of $\phi_{ij}$ only continuous.
\begin{prop} \label{p-contractive} Let $x(t)=(x_1(t),x_2(t),\ldots,x_N(t))$ be a solution to \eqref{eq:HK}, with  $\phi_{ij}\in \C([0,1],]0,+\infty[)$,
in any of the senses given in Definition \ref{def:ODE-sol}, and $0\leq T^1<T^2$. It then holds
\begin{equation}
\overline{co}\left(\left\{x_1(T^1),x_2(T^1),\ldots,x_N(T^1)\right\}\right)\supseteq \overline{co}\left(\left\{x_1(T^2),x_2(T^2),\ldots,x_N(T^2)\right\}\right).\label{e-inclusion}
\end{equation}
\end{prop}
\begin{proof}  Let $x(\cdot)$ be a given Krasovskii solution and define the set $X(t):=\overline{co}\left(\left\{x_1(t),x_2(t),\ldots,x_N(t)\right\}\right)$. Also define the sets 
$$A(T^1):=\left\{T^2\in (T^1,+\infty)\mbox{~~s.t.~~} X(T^1)\not\supseteq X(T^2)\right\}.$$

The statement can be reformulated as following: for all $T^1\geq 0$ the set $A(T^1)$ is empty. We prove it by contradiction: assume that there exists $A(T^1)$ nonempty. Since it is bounded from below by $T^1$ itself, it admits an infimum $T^3\geq T^1$. We now prove the following:

\noindent {\bf Claim a)} It either holds $\inf(A(T^1))=T^1$ or $\inf(A(T^3))=T^3$.

The claim is proved as follows: if $T_3=T_1$, the first statement holds by construction. Otherwise, it holds $X(T^1)\supseteq X(T^3)$, since the condition is closed. Take a sequence $T^2_k\in A(T^1)$ with $T^2_k\searrow T^3$ and observe that  $X(T^1)\not \supseteq X(T^2_k)$ implies $X(T^3)\not \supseteq X(T^2_k)$. Thus $T^2_k\in A(T^3)$ for all $k$, hence $T^3=\inf(A(T^3))$.

Thanks to Claim a), by renaming $T^3=0$ or $T^1=0$, we assume $\inf(A(0))=0$ from now on. Take now a sequence of times $t_k\searrow 0$ such that there exists $i=1,\ldots, N$ for which $x_i(t_k)\not\in X(0)$. Since the number of agents is finite, eventually passing to a subsequence, there exists a single agent (that we relabel as agent 1) satisfying $x_1(t_k)\not\in X(0)$. By continuity of $x_1(t)$, it holds $x_1(0)\in \partial X(0)$, that is the boundary of $X(0)$. Since $X(0)$ is a $n$-dimensional convex polyhedron, there exists a small ball $B(x_1(0),\eps)$ and a finite number of hyperplanes passing through $x_1(0)$ identified by outer unitary vectors $\nu_1,\ldots \nu_j$ such that:
\begin{itemize}
\item for all $x\in X(0)$ it holds $(x-x_1(0))\cdot \nu_l\leq 0$ for all $l=1,\ldots,j$;
\item for all $x\in B(x_1(0),\eps)\setminus X(0)$ it holds $(x-x_1(0))\cdot \nu_l>0$ for at least one index $l=1,\ldots,j$.
\end{itemize}
Since chosen unitary vectors are in finite number, eventually passing to a subsequence of $t_k$, one can select a single unitary vector (denoted simply as $\nu$ from now on) such that $(x_1(t_k)-x_1(0))\cdot \nu>0$ for all $t_k$.

We now define the functions $f_i:=(x_i(t)-x_1(0))\cdot \nu$, that are absolutely continuous, and $f(t):=\max_{i=1,\ldots,N}f_i(t)$, that is the maximum of a finite number of absolutely continuous functions, hence absolutely continuous itself. Since $f(t_k)\geq f_1(t_k)>0$, for any choice of $\eps'>0$ the set $A_{\eps'}:=(f(t)>0)\cap (0,\eps')$ is nonempty. For almost every $t\in A_{\eps'}$, one has that $f,f_1,\ldots, f_N$ are differentiable. Observe that, if $x_i(t)$ realizes $f(t)=f_i(t)$, then for each $j\neq i$ it holds
\begin{equation}
(x_j(t)-x_i(t))\cdot \nu=(x_j(t)-x_1(0))\cdot \nu+(x_1(0)-x_i(t))\cdot \nu=f_j(t)-f_i(t)\leq f(t)-f(t)=0.\label{e-jineg}
\end{equation}

We now compute $\dot f_i(t)$ for $t$ such that $f(t)=f_i(t)>0$ and $f_i(t)$ is differentiable. Since the Krasovskii multifunction satisfies \eqref{e:HKF}, there exist $\alpha_j\in[0,1]$ such that 
$$\dot f_i(t)=\dot x_i(t)\cdot \nu=\sum_{j\neq i : \|x_i-x_j\|=1} \alpha_j \phi_{ij}(1) (x_j(t)-x_i(t))\cdot \nu+ \sum_{j\neq i : \|x_i-x_j\|<1} \phi_{ij}(\|x_i-x_j\|) (x_j(t)-x_i(t))\cdot \nu\leq 0.$$
Here we used \eqref{e-jineg} and positivity of $\phi_{ij}$. By Danskin Theorem \cite{danskin}, it holds $\dot f=\max_{i\mbox{~s.t.~}f(t)=f_i(t)}\dot f_i(t)$, hence $\dot f\leq 0$ whenever $f>0$ and it is differentiable. This implies that $f$ is never strictly positive. This contradicts $f(t_k)>0$. Thus, for the chosen Krasovskii solution, \eqref{e-inclusion} holds. 

Since the proof holds for any Krasovskii solution, the statement holds for any definition of solution, by recalling Proposition \ref{p-inclusion} above.
\end{proof}

\section{The linear Hegselmann-Krause model in $\R$}
\label{sec:HK-R}
Here we focus on the Hegselmann-Krause model \eqref{eq:HK} with $n=1$, $\phi_{ij}\equiv 1$, i.e.
on the linear case in dimension one. Even in this simplified setting, the set of solutions is highly dependent on the choice of the  definition.
Moreover, uniqueness and properties P1-2-3) may fail.

\subsection{A toy example: two agents} \label{sec:toy-ex}
The simplest non-trivial example of \eqref{eq:HK} is given by the case $n=1$, $N=2$ and $\phi_{ij}\equiv 1$, i.e. by the system:
\begin{equation}\label{eq:HK1d2a}
\dot{x_1}=\chi_{|x_1-x_2|<1} (x_2-x_1), \qquad
\dot{x}_2=\chi_{|x_1-x_2|<1} (x_1-x_2), 
\end{equation}
where $\chi$ is the indicator function.
We consider the initial condition:
\begin{equation}\label{eq:ic-toy}
x(0)=(x_{1,0},x_{2,0}).
\end{equation}
For initial conditions 
such that $|x_{1,0}-x_{2,0}|\not=1$,
the solution is unique (for all considered concepts): constant for the case $|x_{1,0}-x_{2,0}|>1$ and
verifying:
\begin{equation}\label{eq:exp-conv-sol}
x_1(t)-x_2(t)=e^{-2t} (x_{1,0}-x_{2,0}),\qquad
x_1(t)+x_2(t)=x_{1,0}+x_{2,0},
\end{equation} for the case $|x_{1,0}-x_{2,0}|<1$.\\
To deal with the case $|x_{1,0}-x_{2,0}|=1$, we first distinguish two
possible stratifications, both based on the stratified set:
\begin{equation}\label{eq:str-set-2a}
\Gamma=M_1\cup M_2 \cup M_3
\end{equation}
with $M_1=\{(x_1,x_2): |x_1-x_2|<1\}$,
$M_2=\{(x_1,x_2): |x_1-x_2|>1\}$ and
$M_3=\{(x_1,x_2): |x_1-x_2|=1\}$. 
The first is given by $S_1=(\Gamma, \{1,2,3\},\emptyset,\emptyset)$,
and the second by $S_2=(\Gamma, \{1,2\},\{3\},\Sigma)$
with $\Sigma(3)=1$.
We have  the following:
\begin{prop}\label{prop:toy-ex}
Consider the Cauchy problem \eqref{eq:HK1d2a}-\eqref{eq:ic-toy} with
$|x_{1,0}-x_{2,0}|=1$. Then, the following holds:
\begin{itemize}
\item[i)] The only classical solution is the constant one
$x(t)\equiv x(0)$.
\item[ii)] There is an infinite number of Caratheodory  solutions 
parameterized by $\bar{t}$: constant on the interval $[0,\bar{t}[$,
and for $t\geq \bar{t}$ given by
\[
x_1(t)-x_2(t)=e^{-2(t-\bar{t})} (x_{1,0}-x_{2,0}),\qquad
x_1(t)+x_2(t)=x_{1,0}+x_{2,0}.
\]
\item[iii)] Filippov (and Krasovsky) solutions coincide with Caratheodory
solutions. 
\item [iv)] The unique CLSS  solution is the constant one.
\item[v)] The only stratified solution for $S_1$ is the constant solution,
while the only one for $S_2$ is \eqref{eq:exp-conv-sol}.
\item[vi)] There is no robust solution.
\item[vii)] The constant solution and \eqref{eq:exp-conv-sol}
are the only cone-robust solutions (for any definition for which they
are solutions).
\end{itemize}
In particular  the only concept of solution for which \eqref{eq:exp-conv-sol} is the unique solution is that of stratified solution for the stratification $S_2$.
\end{prop}

\begin{remark} \label{r-2agents}
If we consider the variant model \eqref{e:HK1}, then \eqref{eq:exp-conv-sol} is
the only classical, Caratheodory, CLSS and stratified solution.
This special situation of uniqueness does not occur for more than two
agents, see Section \ref{sec:three-inR}.
\end{remark}

\begin{proof}
Let us start with Filippov solutions. If $x(\cdot)$ is a solution,
we have $\frac{d\,|x_1-x_2|}{dt}(t)\leq 0$ for almost every $t$, thus
we can define $\bar{t}=\inf\{t: |x_1(t)-x_2(t)|<1\}$, possibly $\bar{t}=+\infty$.
For $t\geq \bar{t}$, $x(\cdot)$ is a solution to a linear ODE, thus it is unique.
This shows that Fillippov (and Krasovsky) solutions are those given by ii). Since they satisfy \eqref{eq:HK1d2a} for almost every time, they are also Caratheodory solutions. This proves ii) and iii).\\
The only Caratheodory solution satisfying \eqref{eq:HK1d2a} for all times
is the constant one, thus i) is proved. Similarly, each sample-and-hold solution
is constant, thus iv) is proved.\\
For $S_1$ the cell $M_3$ is of type I, thus the constant solution is the stratified one, while for $S_2$ the cell $M_3$ is of type two and the solution must enter $M_1$, thus it coincides with \eqref{eq:exp-conv-sol}. This proves v).\\
Notice that, if we perturb the initial datum so that $|x_{1,0}-x_{2,0}|>1$
then the only solution is the constant one, while if perturb the initial datum so that $|x_{1,0}-x_{2,0}|<1$ then \eqref{eq:exp-conv-sol} is the only solution.
This proves vi) and vii).
\end{proof}

For what concerns the solution properties P1-2-3),
it is interesting to notice that some properties hold true for all solutions. More precisely, we have the following:
\begin{prop}\label{prop:bar-clu-2a}
Consider the Cauchy problem \eqref{eq:HK1d2a}-\eqref{eq:ic-toy} ,
then the following holds. The properties P1) and P2) hold for all solutions.
Property P3) only holds for classical, CLSS and stratified solutions.
\end{prop}

\begin{proof}
The proof follows directly from Proposition \ref{prop:toy-ex}.
\end{proof}

\subsection{The case of 3 agents in $\R$}\label{sec:three-inR}
The toy example of Section \ref{sec:toy-ex} is the minimal nontrivial example one can build. 
Uniqueness of solution is already lost,
however the set of solutions is given by a one-parameter family and some properties, such as invariance of the barycenter, still hold true.
In this section we consider three agents in $\R$, still with linear dynamics,
showing more complexity and a complete loss of such properties. 

We consider the dynamics \eqref{eq:HK} for $n=1$, $N=3$ and $\phi_{ij}=1$. The system reads as
\begin{equation}\label{eq:HK1d3a}
\left\{
\begin{array}{ll}
\dot{x_1}=\chi_{|x_1-x_2|<1} (x_2-x_1)+\chi_{|x_1-x_3|<1} (x_3-x_1), &\qquad
x_1(0)=x_{1,0},\\
\dot{x}_2=\chi_{|x_1-x_2|<1} (x_1-x_2)+\chi_{|x_2-x_3|<1} (x_3-x_2), &\qquad
x_2(0)=x_{2,0},\\
\dot{x_3}=\chi_{|x_1-x_3|<1} (x_1-x_3)+\chi_{|x_2-x_3|<1} (x_2-x_3), &\qquad
x_3(0)=x_{3,0}.\\
\end{array}
\right. 
\end{equation}

Notice that we can always change the order of the agents
and apply a translation,
thus we will assume $x_{1,0}\leq x_{2,0}=0\leq x_{3,0}$. We will distinguish the following Initial Conditions (IC for short) cases:
\begin{enumerate}[{IC}-A)]
\item   $x_{1,0}<-1$ and $x_{3,0}>1$;
\item $x_{1,0}>-1$ and $x_{3,0}<1$;
\item  $x_{1,0}=-1$ and $x_{3,0}>1$ (or
the symmetric case $x_{1,0}<-1$ and $x_{3,0}=1$);
\item $x_{1,0}=-1$ and $0<x_{3,0}<1$
(or the symmetric case $-1<x_{1,0}<0$ and $x_{3,0}=1$);
\item The most interesting case \eqref{e:E} is when initial distances are exactly 1, i.e. \begin{equation}x_{1,0}=-1,\qquad x_{2,0}=0,\qquad x_{3,0}=1.\tag{IC-E}\label{e:E}\end{equation}
\end{enumerate}

We have the following result for cases IC-A,B,C. See Figure \ref{fig-ABC}.

\begin{prop} \label{p-nonex-class}
The unique solution to IC-A (for any concept of solution in Definition \ref{def:ODE-sol}) is the constant one:

$$x_1(t)=x_{1,0},\qquad x_2(t)=x_{2,0},\qquad x_3(t)=x_{3,0}.$$

The unique solution to IC-B (for any concept of solution in Definition \ref{def:ODE-sol}, except for classical solution) is the one in which all agents exponentially converge to the barycenter, that is invariant. Classical solutions for $x_{3,0}-x_{1,0}\geq 1$ do not exist.

The unique solution to IC-C (for any concept of solution in Definition \ref{def:ODE-sol}) is the one in which $x_1,x_2$ exponentially converge to $\frac{x_{1,0}+x_{2,0}}2$ and $x_3$ is constant. For the symmetric case, $x_1$ is constant and $x_2,x_3$ exponentially converge to $\frac{x_{2,0}+x_{3,0}}2$.
\end{prop}
\begin{proof} The proof is straightforward, by direct computation. Moreover, uniqueness of the Caratheodory solution in the IC-B case with the additional constraint $x_{3,0}-x_{1,0}\geq 1$ ensures the non-existence of a classical solution: indeed, if a classical solution exists, then it coincides with the Caratheodory one, that is not $C^1$ in this case.
\end{proof}

\begin{figure}[h]
\begin{center}
\includegraphics[scale=.35]{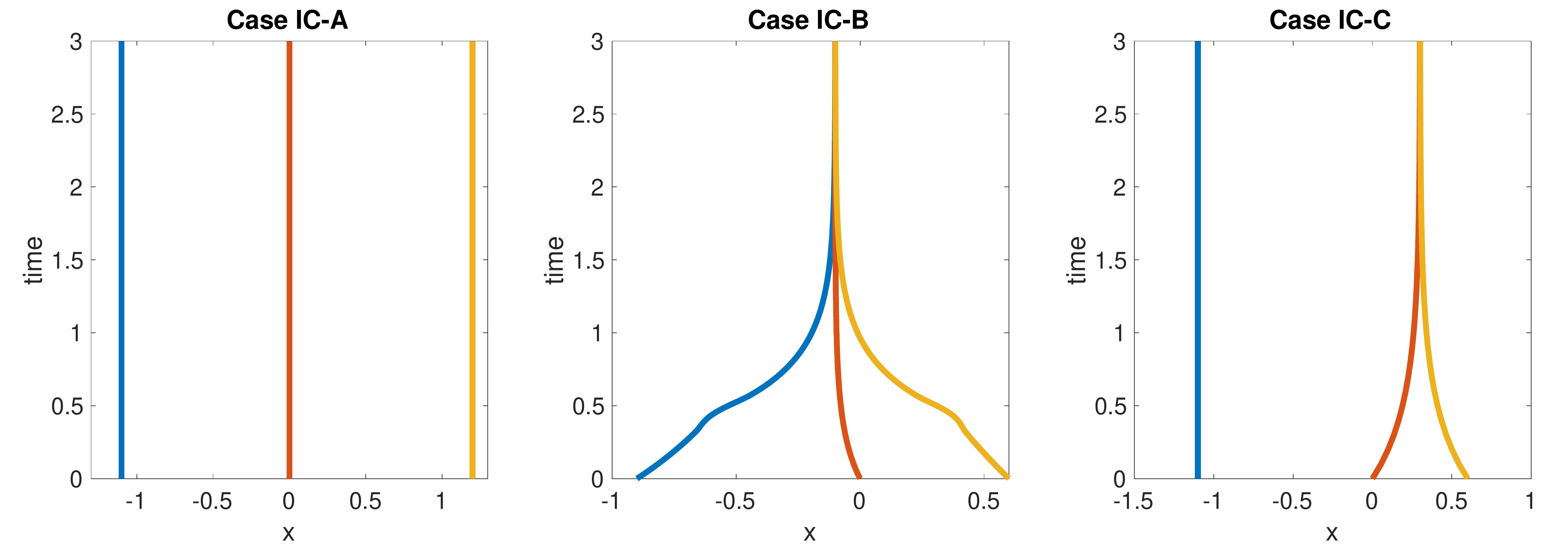}
\caption{Solutions for Initial Condition A (left), B (center), C (right).}
\label{fig-ABC}
\end{center}
\end{figure}

In the remainder we focus on case \eqref{e:E}, as case IC-D and its symmetric are treatable as a sub-case.
We first define four special trajectories: 
$x^\alpha,x^\beta,x^\gamma,x^\delta$. The first trajectory $x^\alpha$ is the constant solution:
\begin{equation}\label{eq:const-sol-3a}
x^\alpha_1(t)\equiv x_{1,0},\quad
x^\alpha_2(t)\equiv x_{2,0},\quad
x^\alpha_3(t)\equiv x_{3,0}.
\end{equation}
The second trajectory $x^\beta$ is the one exponentially converging to the barycenter:
\begin{equation}\label{eq:exp-conv-sol-3a}
x^\beta_1(t)=-e^{-t}\chi_{[0,\ln(2)[} + \frac{-e^{-3(t-\ln(2))}}{2}\chi_{[\ln(2),+\infty[},\ 
x^\beta_2\equiv 0,\ 
x^\beta_3(t)=e^{-t}\chi_{[0,\ln(2)[} + \frac{e^{-3(t-\ln(2))}}{2}\chi_{[\ln(2),+\infty[}.
\end{equation}
The third trajectory $x^\gamma$ has the first two agents exponentially converging and the third
constant:
\begin{equation}\label{eq:exp-conv-sol-3a-12}
x^\gamma_1(t)=-\frac{1+e^{-2t}}{2},\ x^\gamma_2=\frac{e^{-2t}-1}{2},\ x^\gamma_3\equiv 1.
\end{equation}
Finally, the fourth trajectory $x^\delta$ has the second and third agents exponentially converging and the first constant:
\begin{equation}\label{eq:exp-conv-sol-3a-23}
x^\delta_1(t)\equiv -1,\  x^\delta_2=\frac{1-e^{-2t}}{2},\
x^\delta_3(t)= \frac{1+e^{-2t}}{2}.
\end{equation}

\subsubsection{Caratheodory and Filippov solutions}
We now study the family of Caratheodory solutions with initial data \eqref{e:E}. We have the following:
\begin{prop} \label{p-Cprop}
Consider the Cauchy problem \eqref{eq:HK1d3a} in case \eqref{e:E}. Then the following holds:
\begin{itemize}
\item[i)] The set of Caratheodory
solutions is given by the union of three one-parameter families
parameterized by $\bar{t}\in [0,+\infty]$: 

\begin{equation}\begin{cases}
x^\alpha(t) &\mbox{~~for~}t\in [0,\bar{t}[,\\
x^i(t-\bar{t}), \mbox{~~with~}i=\beta,\gamma,\delta&\mbox{~~for~}t\geq \bar{t}.
\end{cases}\label{e-Csol}\end{equation}

\item[ii)] For the modified model \eqref{e:HK1}, the set of Caratheodory solutions
is given by $\{x^i(t)\ :\ i=\beta,\gamma,\delta\}$.
\item[iii)] All Caratheodory solutions satisfy P1-2), while P3) fails, even for the modified model \eqref{e:HK1}.
In particular the final clusters' number and positions depend on the solution:
3 clusters for $x^\alpha$, 1 cluster for $x^\beta$ and 2 clusters for
$x^\gamma$ (in positions $-\frac12$ and $1$) and
$x^\delta$ (in positions $-1$ and $\frac12$).
\end{itemize}
\end{prop}
See a representation of Caratheodory solutions in Figure \ref{fig-CarE}.
\begin{figure}[h]
\begin{center}
\includegraphics[scale=.35]{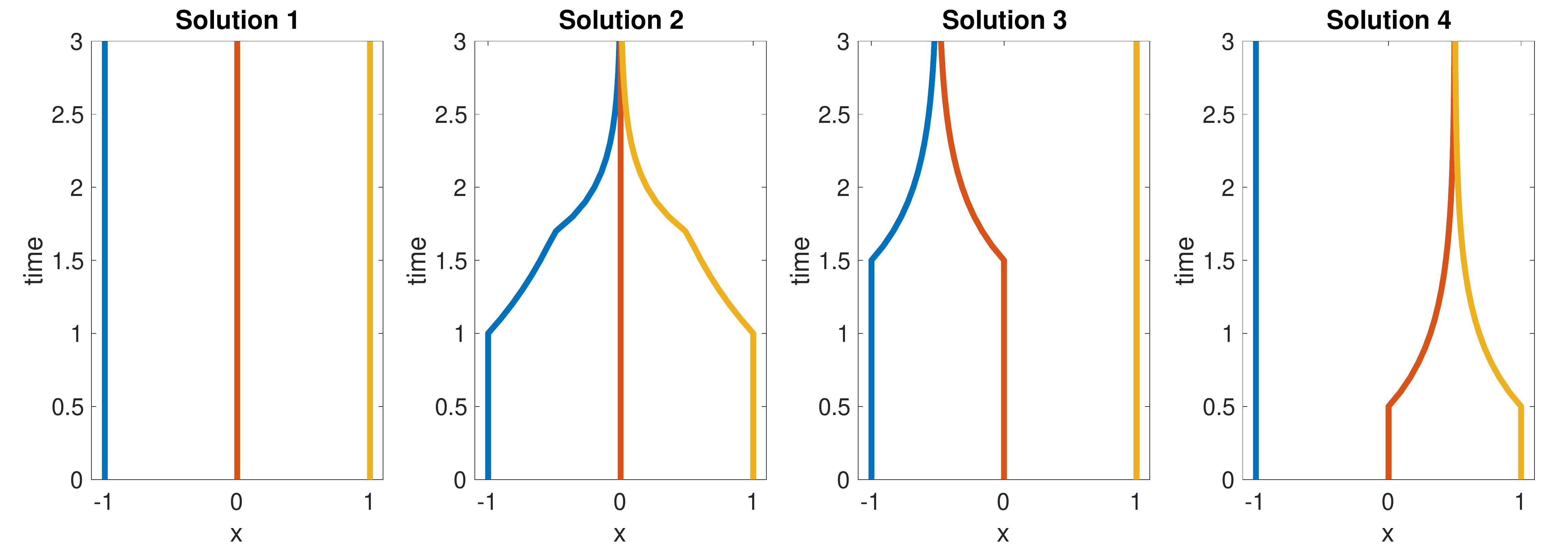}
\caption{Caratheodory solutions for Initial Condition E.}
\label{fig-CarE}
\end{center}
\end{figure}

\begin{proof} It is easy to prove that all trajectories given in \eqref{e-Csol} are Caratheodory solutions of \eqref{eq:HK1d3a} with initial data \eqref{e:E}. We now prove that there exists no other solution. With this goal, we first prove the following two claims:
\begin{description}
\item[Claim a)] If there exists $\bar{t}$ such that $|x_2(\bar{t})-x_1(\bar{t})|<1$ then
for all $t\geq\bar{t}$ we have $|x_2(t)-x_1(t)|<1$. The same result holds for $x_2$ and $x_3$.
\end{description}

We prove the claim by contradiction. Assume that $|x_2(\bar{t})-x_1(\bar{t})|<1$ and $|x_2(\tilde t)-x_1(\tilde t)|\geq 1$ for some $\tilde t>\bar t$. 
Since $x_1$ and $x_2$ are differentiable almost everywhere there exists $t\in (\bar t,\tilde t)$ such that $|x_2(t)-x_1(t)|\in\left(\frac12,1\right)$, and $|x_1(\cdot)-x_2(\cdot)|$
is differentiable at $t$ with strictly positive derivative. Then, it holds:
$$
\frac{d}{dt}(x_2(t)-x_1(t))=-2(x_2(t)-x_1(t))+\chi_{|x_3(t)-x_2(t)|<1} (x_3(t)-x_2(t))
\leq -2(x_2(t)-x_1(t))+1<-2\frac12 +1=0.
$$
This leads to a contradiction. The claim is proved.
\begin{description}
\item[Claim b)] If there exists $\bar{t}$ such that  $|x_2(\bar{t})-x_1(\bar{t})|>1$ then
for all $t\geq\bar{t}$ we have $x_1(t)=x_1(\bar{t})$ and
similarly for $x_2$ and $x_3$.
\end{description}
The proof is easy: notice that $\dot{x}_1=0$ and $\dot{x}_2\geq 0$ for almost every $t\geq\bar{t}$. This implies that $|x_2(\bar{t})-x_1(\bar{t})|$ is increasing, thus the claim is proved.

We now define the times:
\[
t_{12}=\inf \{t: |x_1(t)-x_2(t)|<1\},\qquad
t_{23}=\inf \{t: |x_2(t)-x_3(t)|<1\},
\]
possibly equal to $+\infty$ when sets are empty, and prove the following:
\begin{description}
\item[Claim c)]If $0<t_{12},t_{23}<+\infty$ then $t_{12}=t_{23}$.
\end{description}
Assume, by contradiction, that $t_{12}<t_{23}$ (the other case
being similar). On the interval $[t_{12},t_{23}]$
we have $|x_2(t)-x_3(t)|\geq 1$ by definition of $t_{23}$. Claim b) ensures that $|x_2(t)-x_3(t)|=1$ for all $t\in [t_{12},t_{23}]$,
otherwise we would have  $t_{23}=+\infty$.

Take now the definition of $t_{12}$ and apply Claim a), that ensures that $|x_2(t)-x_1(t)|<1$ for all $t>t_{12}$. Merging it with $|x_2(t)-x_3(t)|=1$ on the interval $[t_{12},t_{23}]$, we have both $\dot x_2(t)<0, \dot{x}_3(t)=0$. This contradicts $|x_2(t)-x_3(t)|=1$  on the same interval. This proves the claim.

We are now ready to prove i). If $t_{12}=t_{23}<+\infty$ then by Claim a), the solution is constant on $[0,t_{12}]$, then given by $x^\beta (t-t_{12})$ on $[t_{12},+\infty[$.
If $t_{12}<t_{23}=+\infty$, then the solution is constant on $[0,t_{12}]$ then given by $x^\gamma(t-t_{12})$ on $[t_{12},+\infty[$.
Similarly, if $t_{23}<t_{12}=+\infty$ then the solution is constant on $[0,t_{23}]$ then given by $x^\delta(t-t_{23})$ on $[t_{12},+\infty[$.
In the last case $t_{12}=t_{23}=+\infty$, from Claims a) and b) we deduce 
$\dot{x}_1\equiv \dot{x}_2\equiv \dot{x}_3\equiv 0$, thus the solution
is the constant one $x^\alpha$. This proves i).

To prove ii), it is enough to notice that the constant solution is no more a Caratheodory solution. Finally, iii) follows directly by i) and ii).
\end{proof}

\begin{remark} One might expect that solutions to \eqref{eq:HK} with $\phi_{ij}(r)=1$ exhibit uniform exponential convergence to their limit, in the following sense: there exist $C,k>0$ such that for any trajectory $x(t)$ it holds 
\begin{equation}\label{e-exp}
\|x(t)\|\leq Ce^{-kt} \|x(0)\|.
\end{equation}
Indeed, beside the points in which $x(t)$ crosses $\mathcal{M}$, the dynamics is linear. Yet, exponential convergence does not hold for Caratheodory solutions, as Proposition \ref{p-Cprop} shows. Indeed, given the initial condition \eqref{e:E}, one can wait an arbitrarily long time $\bar t$ before starting exponential convergence to $0$. Thus, a global constant $C$ in \eqref{e-exp} does not exist.\\
This also shows that Filippov-Krasovskii solutions do not satisfy exponential convergence either, due to Proposition \ref{p-inclusion}.
\end{remark}

\subsubsection{Filippov solutions}
We now study the family of Filippov solutions with initial data \eqref{e:E}. 
Besides Caratheodory solutions studied above, we look for solutions $x(\cdot)$
such that $x_2(t)=x_1(t)+1$ and $x_2(t)> x_3(t)-1$
for all times $t>0$. If such a Filippov solution exists on an interval $[0,T]$, 
then $x(\cdot)$ must satisfy
\begin{equation}
\dot{x}_1(t) = \alpha(t) ,\
\dot{x}_2(t)=- \alpha(t) +(x_3(t)-x_2(t)),\
\dot{x}_3(t)=(x_2(t)-x_3(t))
\end{equation}
for some measurable functions $\alpha :[0,T]\to [0,1]$. The condition $x_2(t)=x_1(t)+1$ for all times implies
\begin{equation}
\alpha(t)=\frac{x_3(t)-x_2(t)}{2}.
\end{equation}
Defining $y(t)=x_3(t)-x_2(t)$, we get $\dot{y}(t)=-\frac{3}{2}y(t)$,
thus $y(t)=e^{-\frac{3}{2}t}$ and the solution is given by\FRc{soluzione corretta}:
\begin{equation}\label{eq:Fil-sol-3ag-R}
x_1(t)=-\frac{2}{3}-\frac{1}{3}e^{-\frac{3}{2}t},\
x_2(t)=\frac{1}{3}-\frac{1}{3}e^{-\frac{3}{2}t},\
x_3(t)= \frac{1}{3}+\frac{2}{3}e^{-\frac{3}{2}t}.
\end{equation}
We get the following:
\begin{prop} \label{p-Fsol}
Consider the Cauchy problem \eqref{eq:HK1d3a} with initial data \eqref{e:E}. The set of Fillippov solutions contains the set of Caratheodory solutions
and the following two-parameters families.\\
Given $0\leq t_1<t_2\leq +\infty$ define a solution $z(\cdot)$ as follows.
On the interval $[0,t_1]$ the solution is constant, on 
the interval $[t_1,t_2]$ the solution is given by $z(t)=x(t-t_1)$ for
$x(\cdot)$ given by \eqref{eq:Fil-sol-3ag-R}, and 
on the interval $[t_2,+\infty[$ the solution satisfies $z_1(t)\equiv z_1(t_2)$,
while $\dot{z}_2(t)=z_3(t)-z_2(t)=-\dot{z}_3(t)$.
The solution $z$ converge to an asymptotic state with the first agent
at $\bar{x}_1 \in [-1,-\frac{2}{3}]$ and the other two at
$\bar{x}_2 = -\frac{\bar{x}_1}{2}$.\\
Given $0\leq t_1<t_2\leq +\infty$ define a solution $w(\cdot)$ as follows.
On the interval $[0,t_1]$ the solution is constant, on 
the interval $[t_1,t_2]$ the solution is given by $w(t)=x(t-t_1)$ for
$x(\cdot)$ given by \eqref{eq:Fil-sol-3ag-R}, and 
on the interval $[t_2,+\infty[$ all agents interact converging to zero.\\
Similarly we can define other two-parameters families by symmetry
exchanging the roles of agent $1$ and $3$.\\
\end{prop}
\begin{proof}
Claims a) and b) of Proposition \ref{p-Cprop} hold true for Filippov solutions
using the same proof. With notations as in Claim c), assume
that $t_{12}<t_{23}$ then again we conclude $|x_2(t)-x_3(t)|=1$
and $|x_1(t)-x_2(t)|<1$ on $[t_{12},t_{23}]$. Therefore the solution
on the interval $[t_{12},t_{23}]$ is given by $x(t-t_{12})$, with
$x(\cdot)$ given by \eqref{eq:Fil-sol-3ag-R}.
The other claims easily follow.
\end{proof}

In Figure \ref{fig-FilE}, we depict representatives for Filippov solutions described in Proposition \ref{p-Fsol}.

\begin{figure}[h]
\begin{center}
\includegraphics[scale=.4]{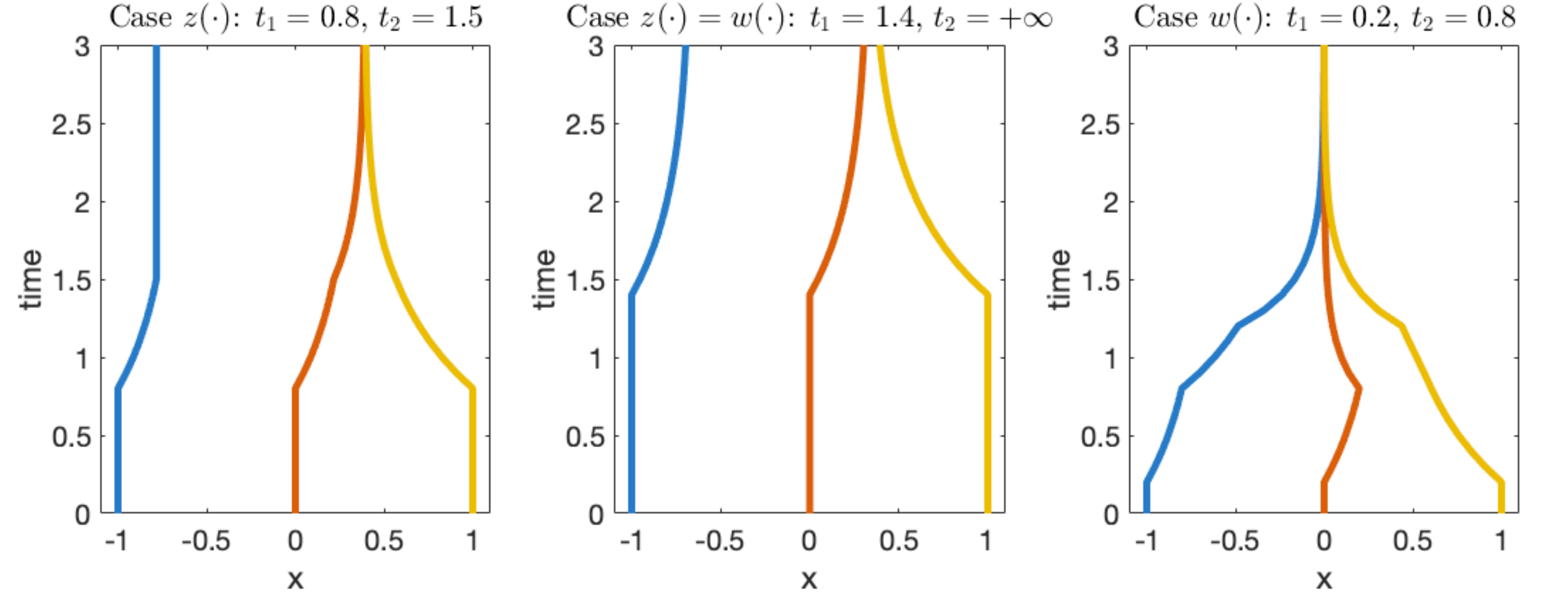}
\caption{Filippov solutions for Initial Condition E.}
\label{fig-FilE}
\end{center}
\end{figure}

\subsubsection{Stratified solutions}
In this Section we focus on stratified solutions. The latter are unique
for a given stratification, but the stratification is not unique.
In particular, the final number of clusters is dependent 
on the initial datum but also on the chosen stratification.
Here, we build a stratification ensuring the minimal number of clusters
in the final configuration for any initial datum.

The construction of the stratification is  based on a careful analysis of singularities. Since stratified solutions satisfy the equation \eqref{eq:HK1d3a} for almost
every time, then the barycenter $\bar{x}$ is invariant. By eventually applying a translation, we assume $\bar{x}=0$ from now on. It thus holds
\begin{equation}\label{eq:x3-as-funct}
x_3(t)=-x_1(t)-x_2(t).
\end{equation}
The problem of finding a stratification can be solved on 
$\R^2$, as the stratification in $\R^3$ can be obtained by using \eqref{eq:x3-as-funct}. Define the following lines for $i,j\in\{1,2,3\}$, $i\not= j$:
\begin{equation}\label{eq:lines}
l_{ij}^{\pm}=\{(x_1,x_2):x_i=x_j\pm 1\},
\end{equation}
where we used the equality \eqref{eq:x3-as-funct}: this means that $l^{\pm}_{32}$, that are the sets $x_3=x_2\pm 1$, is given
by $x_2=-\frac{x_1}{2}\mp \frac{1}{2}$. Similarly, $l^{\pm}_{31}$ are given
by $x_1=-\frac{x_2}{2}\mp \frac{1}{2}$.
These lines meet at 12 points, with 4 on the coordinate axes. See Figure \ref{fig:strat}.

The following points belong to the first orthant:
\begin{equation}\label{eq:int-point-1}
l_{21}^+\cap l_{31}^- = (0,1),\quad
l_{21}^-\cap l_{32}^- = (1,0),\quad
l_{32}^-\cap l_{31}^- = \left(\frac{1}{3},\frac{1}{3}\right).
\end{equation}
Points in the third orthant are obtained
by symmetry with respect to the origin.\\
The following points belong to the second orthant (two lie on axes, thus they are shared with other orthants):
\begin{equation}\label{eq:int-point-2}
l_{21}^+\cap l_{31}^- = (0,1),\quad
l_{21}^+\cap l_{32}^+ = (-1,0),\quad
l_{21}^+\cap l_{32}^- = \left(-\frac13,\frac23\right),\quad
l_{21}^+\cap l_{31}^+ = \left(-\frac23,\frac13\right),\quad
l_{32}^-\cap l_{31}^+ = (-1,1).
\end{equation}
Points in the fourth orthant are obtained
by symmetry with respect to the origin.\\
We are now ready to define the strata of our stratification.
\begin{definition}\label{def:strat-3a-1d}
The strata $M^0_1,\ldots,M^0_{12}$
of dimension $0$ are given by the points
\eqref{eq:int-point-1}, \eqref{eq:int-point-2}
and their symmetric with respect to the origin.\\
The strata $M^1_1,\ldots,M^1_{30}$ of dimension $1$
are given by the connected components of the lines
defined in \eqref{eq:lines} after removing the strata of
dimension $0$.\\
The strata $M^2_1,\ldots,M^2_{19}$ are given
by the connected components of $\R^2$ after
removing the strata of dimension $0$ and $1$.\\
The strata of dimension $0$ and $1$ are all
of type II. Define $\Sigma(M^0_i)=M^2_j$,
where $M^2_j$ is such that 
$M^0_i\subset\partial M^2_j$ and $M^2_j$ is the 
stratum containing the point of least norm among those
with such property. Similarly,
define $\Sigma(M^1_i)=M^2_j$,
where $M^2_j$ is such that 
$M^1_i\subset\partial M^2_j$ and $M^2_j$ is the 
stratum containing the point of least norm among those
with such property.
\end{definition}
We refer the reader to Figure \ref{fig:strat} for a graphical illustration of the stratification and the dynamics in some of the strata.
\begin{figure}[h]
\begin{center}
\includegraphics[scale=.5]{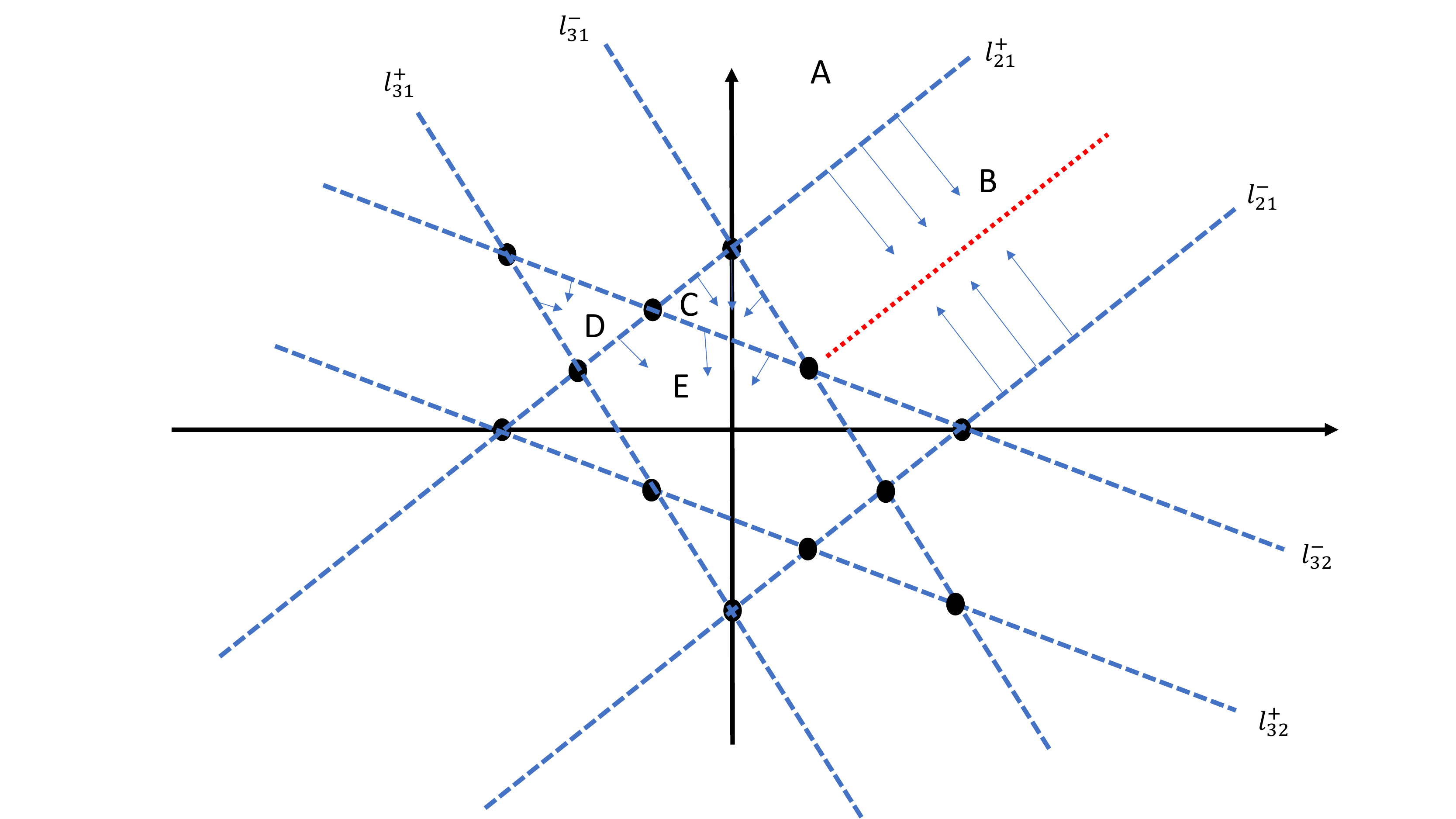}
\caption{Graphical representation of the stratification
given in Definition \ref{def:strat-3a-1d}.}
\label{fig:strat}
\end{center}
\end{figure}

\begin{prop}
Consider the Cauchy problem \eqref{eq:HK1d3a}
and the stratification defined in Definition \ref{def:strat-3a-1d}.
Then, stratified solutions are unique and converge
asymptotically to a configuration with the minimal number of clusters.
\end{prop}
\begin{proof}
Let us start by analyzing the dynamics on the strata of dimension two.\\
{\bf Case A)} There are six unbounded regions where $|x_i-x_j|>1$,
for all pairs  $i,j=1,\ldots,3$, $i\not= j$. See $A$ in Figure \ref{fig:strat}.
The stratified solutions on these regions are constant.\\
{\bf Case B)} There are other six unbounded regions where
$|x_i-x_j|<1$ for only one couple $(i,j)$, $i\not= j$.
The stratified solutions verify $x_i(t)-x_j(t)\to 0$,
while the remaining agent remains fixed. 
For the region $B$ in Figure \ref{fig:strat}, the dynamics satisfies:
\[
\dot{x}_1+\dot{x}_2=0,\quad
\dot{x}_1-\dot{x}_2=-2\, (x_1-x_2),
\]
thus all solutions
tend to the dotted (red) line $x_1=x_2$.\\
{\bf Case C)} There are four bounded regions intersecting
the coordinate axes but not containing the origin.
For the region marked as $C$ in Figure \ref{fig:strat},
the dynamics is given by:
\[
\dot{x}_1=-3\,x_1,\quad
\dot{x}_2=x_1-x_2,
\]
thus solutions exit towards the region marked $E$.\\
{\bf Case D)} There are two bounded regions not intersecting
the coordinate axis. 
For region $D$ in Figure \ref{fig:strat}
the dynamics is given by:
\[
\dot{x}_1=-2\,x_1-x_2,\quad
\dot{x}_2=-2\,x_2-x_1,
\]
and solutions exit towards the region marked $E$.\\
{\bf Case E)} Finally, there is a bounded region containing
the origin, named $E$ in Figure \ref{fig:strat},
where all agents are interacting and trajectories
converge to the origin.

The stratum of dimension one $M^1_i$ have trajectories
exiting to the stratum $\Sigma(M^1_i)$. For instance,
trajectories from $l_{21}^+\cap\partial B$ enter the
region $B$ and the same for $l_{21}^-\cap\partial B$.
Trajectories from $l_{31}^-\cap\partial B$
enter region $C$.\\
Similarly, trajectories from from $l_{32}^-\cap\partial D$ enter the region $D$ and the same for $l_{31}^+\cap\partial D$. Trajectories from $l_{21}^+\cap\partial D$
enter region $E$.

Finally, the stratum of dimension zero $M^0_i$ have trajectories exiting to the stratum $\Sigma(M^1_i)$.
For instance, the trajectory from $(0,1)$ enters region $C$
and the one from $(-1,1)$ enters region $D$.

We are now ready to complete the proof. 
For two dimensional strata the analysis is as follows.
For strata as $A$ there is uniqueness of trajectories
and three final clusters. For strata as $B$, trajectories
never exit the region and converge to two clusters.
For all other strata, trajectories converge to the origin, which corresponds to a unique cluster.\\
For one dimensional strata there are two cases.
If the stratum is at the boundary of a region of type $A$
and a region of type $B$, then trajectories enter the $B$ region and converge to two clusters.
For all other strata, trajectories converge to the origin, which corresponds to a unique cluster.\\
Finally, trajectories from zero dimensional strata converge to the origin, thus a unique cluster.\\
We conclude that all stratified trajectories converge to a configuration with the minimum number of clusters.
\end{proof}

\subsection{Many agents in $\R$: Caratheodory solution}
In this section, we briefly describe the combinatorial complexity of Caratheodory solutions 
for $N$ agents in $\R$
following the dynamics \eqref{eq:HK} with $\phi_{ij}=1$.
From Proposition \ref{prop:HK-Car} we know
that such solutions are also solutions in the sense of Fililppov
and Krasovskii.

Fix $N\in \mathbb{N}\setminus{0}$ and consider an initial condition such that:
\begin{equation}\label{eq:in-con-dist1}
x_{i+1}-x_{i}=1, \quad i=1,\ldots , N-1.
\end{equation}
Such initial conditions form a one-dimensional manifold
in $\R^N$. The results we state are valid for any permutation of the agents numbering, so will hold for the union of $N!$ one-dimensional manifolds.

To compute the combinatorics related to the number of solutions 
we need to introduce some notation. For the fixed number $N$ of agents, we define
the sets:
\begin{equation}
\Delta_1(N)=\{(n_1,\ldots,n_\ell): n_k\in\N, \sum_k n_k=N\},
\end{equation}
\begin{equation}
\Delta_2(N)=\{(n_1,\ldots,n_\ell): n_k\in\N, \sum_k n_k=N\
and\ n_k+n_{k+1}\geq 3\ for\ k=1,\ldots,\ell-1 \}.
\end{equation}
In other words, $\Delta_1(N)$ is formed by the ordered $\ell$-tuple of natural
numbers summing up to $N$, while $\Delta_2(N)$ has the further restriction that no two consecutive numbers are equal to $1$.\\
Given $(n_1,\ldots,n_\ell)\in\Delta_i(N)$, $i=1,2$, we define a partition 
$P=\{P_1,\ldots,P_\ell\}$ of $\{1,\ldots , N\}$ as follows:
\[
P_1:=\{1,\ldots,n_1\},\qquad\mbox{~and~}\qquad P_k=\left\{1+\sum_{h=1}^{k-1} n_h,\ldots,  \sum_{h=1}^{k} n_h  \right\}, \mbox{~~for~}k=2,\ldots \ell
\]
In other words, $P_{k+1}$ are the $n_{k+1}$ numbers following those in $P_1\cup \cdots \cup P_k$.

The corresponding solutions are described in the next propositions. See also a representation in Figure \ref{fig-CarN}.
\begin{figure}[h]
\begin{center}
\includegraphics[scale=.35]{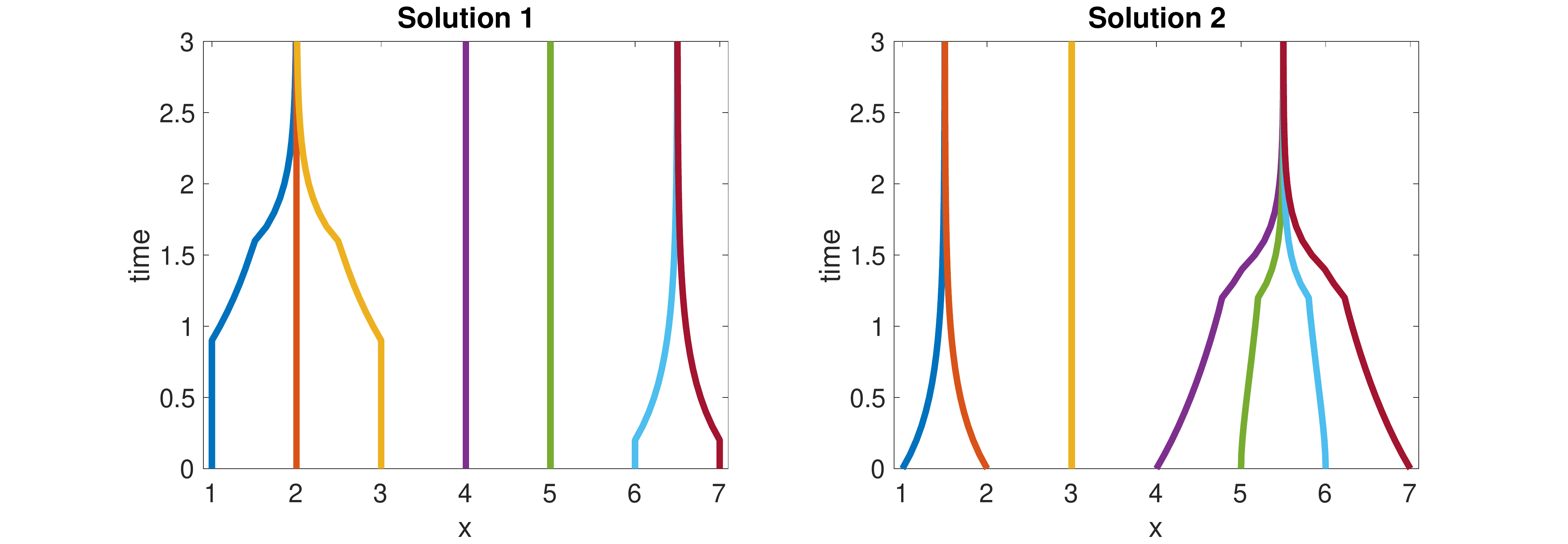}
\caption{Caratheodory solutions corresponding to $(3,1,1,2)\in\Delta_1(7)\setminus\Delta_2(7)$ (Left, see Proposition \ref{prop:HK-multi-R}) and $(2,1,4)\in\Delta_2(7)$ (Right, see Proposition \ref{prop:HK-multi-R-ineq}).}
\label{fig-CarN}
\end{center}
\end{figure}

\begin{prop}\label{prop:HK-multi-R}
Consider the ODE \eqref{eq:HK} with $n=1$, $\phi_{ij}=1$ and an initial condition satisfying \eqref{eq:in-con-dist1}, then the following holds.
For every $(n_1,\ldots,n_\ell)\in\Delta_1(N)$ there exists an 
$\ell$-dimensional parametrized family of distinct Caratheodory solutions converging to a limit $x^{\infty}$ such that $x^{\infty}_i=x^{\infty}_j$
for every $i,j\in P_k$, $k=1,\ldots,\ell$.
\end{prop}

\begin{proof}
Fix $(n_1,\ldots,n_\ell)\in\Delta_1(N)$. 
Given $\bar{t}_1\in [0,+\infty[$ we can define a dynamics for
the first $n_1$ agents: constant on $[0,\bar{t}_1]$ and satisfying:
\[
\dot{x}_1(t)=(x_2-x_1)(t),\quad
\dot{x}_i=(x_{i-1}-x_i)(t)+(x_{i+1}-x_i)(t),\
i=2,\ldots,n_1-1,\quad
\dot{x}_{n_1}(t)=(x_{n_1-1}-x_{n_1})(t),
\]
for $t\geq  \bar{t}_1$. Notice that the first $n_1$ agents
eventually converge to their barycenter.\\
Similarly for every $i$, $i=2,\ldots,\ell$,
given $\bar{t}_i\in [0,+\infty[$  we can define a dynamics for the
$n_i$ agents following $n_1+\ldots+n_{i-1}$: constant on
$[0,\bar{t}_i]$ and with all $n_i$ agents interacting on $[\bar{t}_i,+\infty[$.
All $n_i$ agents will converge to their barycenter. 
We thus proved the statement.
\end{proof}

\begin{prop}\label{prop:HK-multi-R-ineq}
Consider the variant model \eqref{e:HK1}. Let $n=1$, $\phi_{ij}=1$ and an initial condition satisfying \eqref{eq:in-con-dist1}, then solutions are parametrized by $\Delta_2(N)$ as follows.
For each $(n_1,\ldots,n_\ell)\in\Delta_2(N)$ there exists a single Caratheodory solution converging to a limit $x^{\infty}$ such that $x^{\infty}_i=x^{\infty}_j$
if and only $i,j$ belong to the same $P_k$, $k=1,\ldots,\ell$. Moreover,  there is no other Caratheodory solution.
\end{prop}

\begin{proof} Given $(n_1,\ldots,n_\ell)\in\Delta_2(N)$, there exists a Caratheodory solution for \eqref{e:HK1} such that the first $n_1$ agents interact for all times and converge to their barycenter and, in general, the $n_i$ agents following $n_1+\ldots+n_{i-1}$ interact for all times and converge to their barycenter.
This corresponds to solutions constructed in the proof of Proposition \ref{prop:HK-multi-R} for the case $\bar{t}_i=0$, $i=1,\ldots,\ell$.
Moreover, for a group with more than one agent, no Caratheodory solution 
to \eqref{e:HK1} can be constant on a time interval $[0,\bar t]$, $\bar{t}>0$,  as in Proposition \ref{p-Cprop} case ii), due to $a_{ij}(1)=1$ forcing
agents to attract each other if they keep their distance equal to $1$. Similarly, we can not have two consecutive groups consisting of only one agent.\\
This proves the statement.


\end{proof}

\section{Hegselmann-Krause in higher dimensions}
\label{sec:HK-multiR}
As we have seen in Section \ref{sec:HK-R}, solutions may not be unique for all the concepts, except for classical, CLSS and stratified solutions, already in dimension one. Here we show some additional complexity in the set of Caratheodory solutions in higher dimension, as well as loss of uniqueness for CLSS solutions.

First consider \eqref{eq:HK} with $n=2$, $N=4$, $\phi_{ij}=1$ with initial condition:
\begin{equation}\label{eq:id-n2}
x_{1,0}=(0,0),\ x_{2,0}=(1,0),\ x_{3,0}=(1,1),\ x_{4,0}=(0,1).
\end{equation}

\begin{figure}[h]
\begin{center}
\includegraphics[scale=.4]{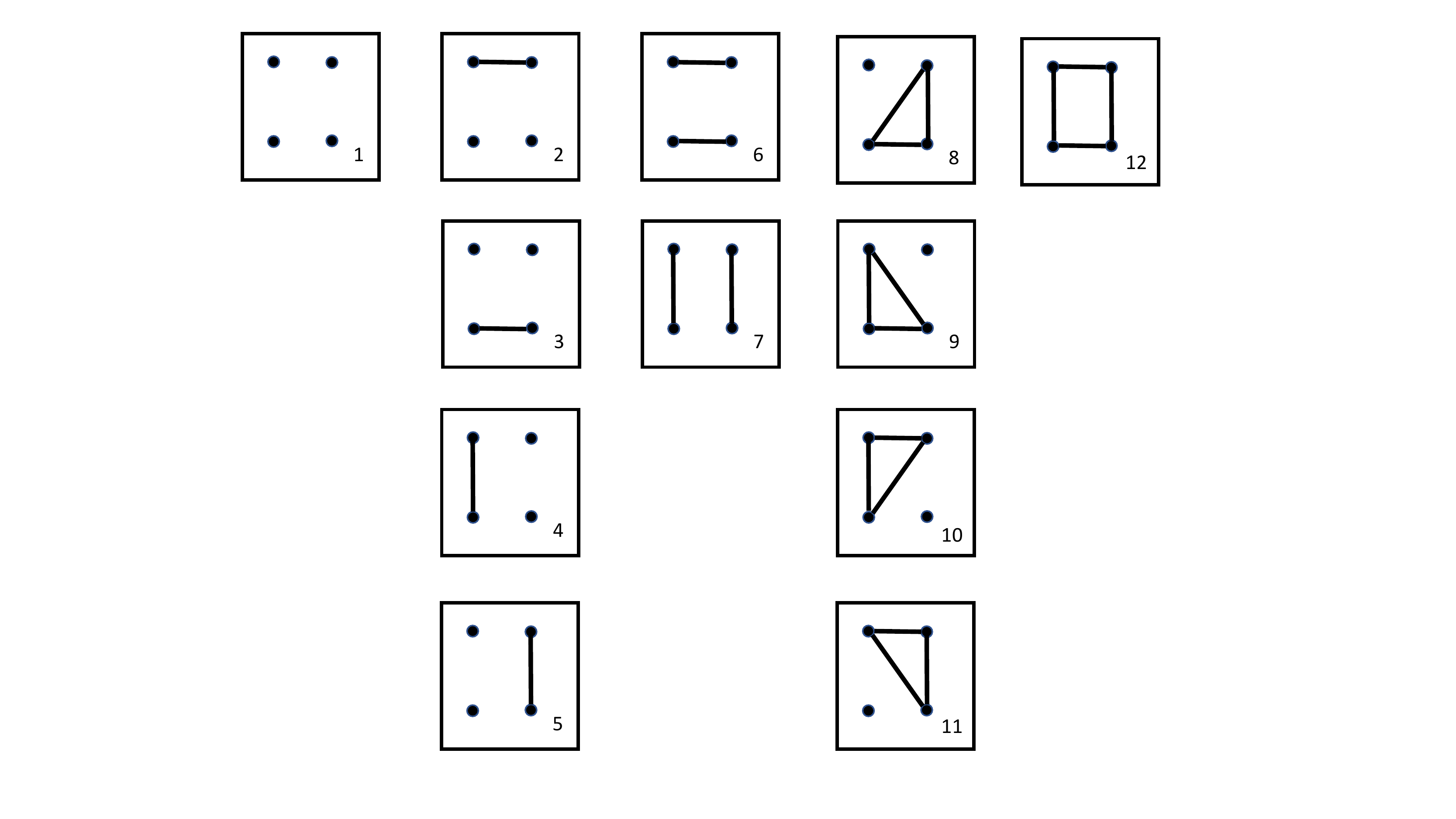}
\caption{Representation of potential family
of solutions for the initial data \eqref{eq:id-n2}.}
\label{fig:2D}
\end{center}
\end{figure}

We have the following:
\begin{prop}
The set of Caratheodory solutions to \eqref{eq:HK}, 
with $n=2$, $N=4$, $\phi_{ij}=1$, and
initial datum \eqref{eq:id-n2},
contains 12 parametric families of solutions as follows.
We refer to Figure \ref{fig:2D} where in each box agents
connected by edges converge to their barycenter:
\begin{itemize}
\item[i)] Case 1: single constant solution;
\item[ii)] Cases 2-5 and 8-12: one-parameter family. Given
$\bar{t}\in [0,+\infty[$ the solution is constant on 
$[0,\bar{t}]$ then connected agents converge to their barycenter;
\item[ii)] Cases 6,7: family parameterized by two two-parameter sets $\TT =A_1\cup A_2$, with 
$A_1=\{(t_1,t_2): 0\leq t_1,t_2\}$, 
$A_2=\{(t_1,t_3): 0\leq t_1,t_3\}$.
If $(t_1,t_2)\in A_1$, then the solution is constant on 
$[0,\min\{t_1,t_2\}]$,  then agents $x_1$ and $x_2$ for case 6
(respectively $x_1$ and $x_4$ for case 7) 
start converging to their barycenter $(\frac12,0)$ 
(respectively $(0,\frac12)$ for case 7) at time $t_1$, while 
agents $x_3$ and $x_4$ for case 6
(respectively $x_2$ and $x_3$ for case 7) 
start converging to their barycenter $(\frac12,1)$
(respectively $(1,\frac12)$ for case 7)  at time $t_2$.
If $(t_1,t_3)\in A_2$, then the solution is as for $A_1$ with $t_1=t_2$ up to time $t_1+t_3=t_2+t_3$, then for $t\geq t_1+t_3$ all agents interact and converge to a unique cluster at $(\frac12,\frac12)$.
\end{itemize}
In particular the number of asymptotic clusters can be 1, 2, 3 or 4. 
There is 1 asymptotic configuration with 4 clusters,
4 asymptotic configurations with 3 clusters,
6 asymptotic configurations with 2 clusters, and
1 asymptotic configuration with 1 cluster.
\end{prop}
\begin{proof}
The proof follows the same arguments as in the proof of Proposition \ref{prop:HK-multi-R}.
\end{proof}
Following the logic of Proposition \ref{prop:HK-multi-R-ineq}, we obtain the following:
\begin{prop} \label{p-Cara1incl}
Consider the variant system \eqref{e:HK1}, 
with $n=2$, $N=4$, $\phi_{ij}=1$.
The set of Caratheodory solutions 
for the initial datum \eqref{eq:id-n2}
has 5 elements, corresponding to cases
8-12 above.
\end{prop}

As pointed out in Section \ref{sec:three-inR}, solutions from the same initial data may converge to different clusters. Here we show that this effect becomes more dramatic in dimension two, with different solutions
converging to arbitrarily far away clusters.
\begin{prop} \label{p:distanziamento}
Consider the ODE \eqref{eq:HK} with $n=2$ and $\phi_{ij}=1$. 
Given $R>0$, there exists a system of $N$ agents with an initial condition such that
there exists two Caratheodory solutions $x^1$, $x^2$ to the Cauchy Problem
with $\lim_{t\to\infty} \sup_i |x^1_i(t)-x^2_i(t)|>R$.
\end{prop}
\begin{proof}
The proof will be constructive, by recursion, by adding agents
at distance $0$, $1$ or bigger than $1$. The solution
$x^1$ is taken to be constant, while $x^2$ is constructed
recursively by making all agents at distance $1$ interact. See a representation in Figure \ref{f-planar}.\\
\begin{figure}[htb]
\begin{center}
\includegraphics[scale=.4]{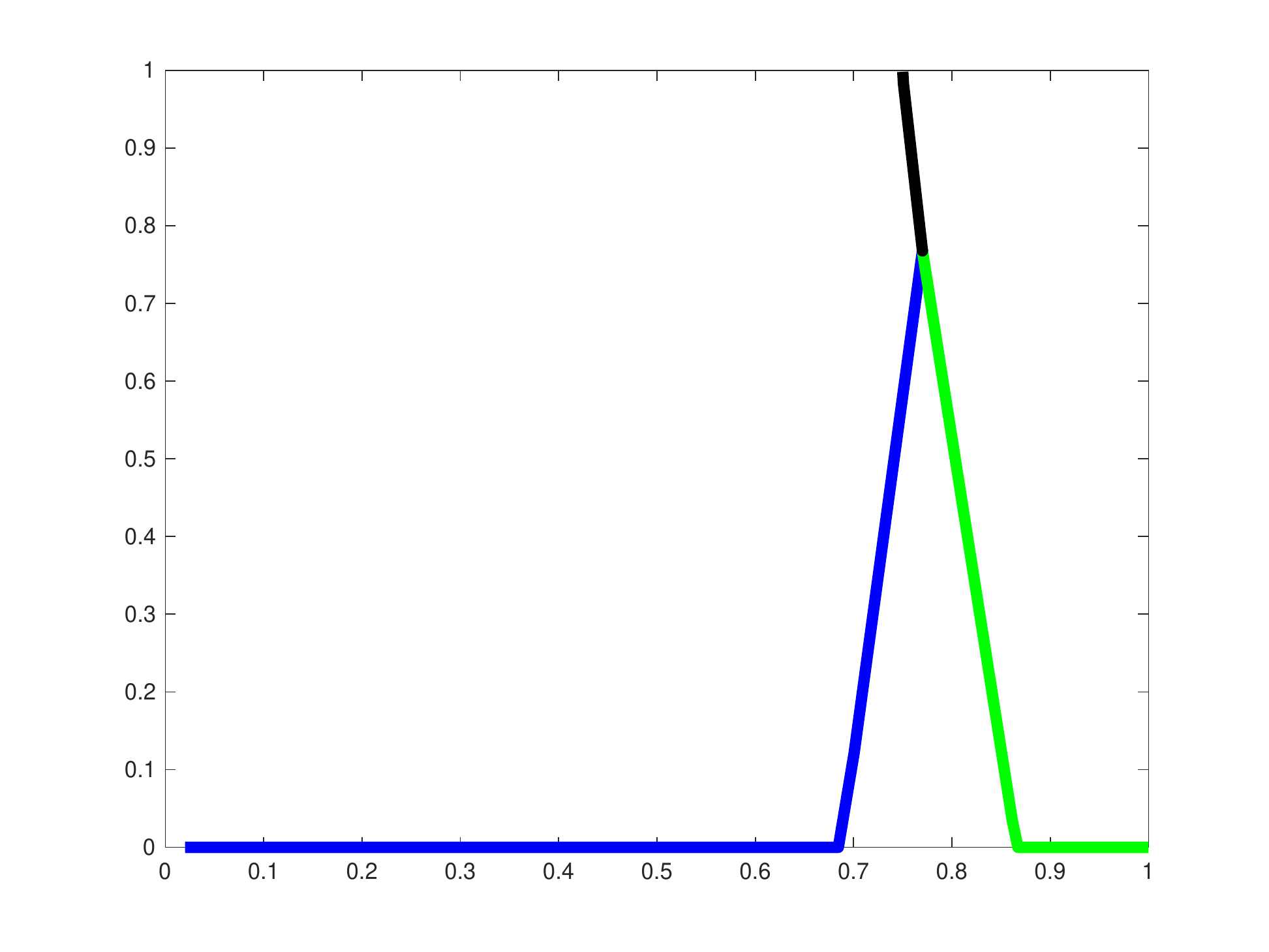}
\caption{The first two steps in the proof of Proposition \ref{p:distanziamento}: 1 agent (blue) starting at (0,0), $N_1=5$ agents (green) starting at (1,0), $N_2=10$ agents (black) starting at (.75,.998).}
\label{f-planar}
\end{center}
\end{figure}

We start with a single agent $x_1$ in position $(0,0)$. Then we add $N_1$ (to be chosen) agents 
with the following initial condition:
\begin{equation}
x_{i,0}=(1,0),\ i=2,\ldots,N_1+1.
\end{equation}
The solution $x^2$ has all agents converging to the
barycenter:
\[
\bar{x}_1=\left(\frac{N_1}{N_1+1},0\right),
\]
which is close to $(1,0)$ for $N_1$ sufficiently big.
Notice that, since the barycenter is invariant,
at every time $t\geq 0$, 
there exists $\epsilon=\epsilon(t)$, with
$0\leq \epsilon\leq \frac{1}{N_1}$,
such that the agents are in the following position:
\begin{equation}\label{eq:bar-time-eps}
x_1=\left(\frac{N_1}{N_1+1}-N_1\,\epsilon,0\right),\qquad
x_{i,0}=\left(\frac{N_1}{N_1+1}+\epsilon,0\right),\ i=2,\ldots,N_1.
\end{equation}
The final distance between the asymptotic state
of the first agent along $x^1$ and $x^2$ is given
by $\frac{N_1}{N_1+1}$, which is close to 1.\\
We now add $N_2>N_1$ (to be chosen) agents in position:
\begin{equation}\label{eq:in-N2}
(y_1,y_2)=
\left(\frac{N_1}{N_1+1}+\epsilon_2\frac{1-N_1}{2},
\sqrt{1-\left(\frac{\epsilon_2(N_1+1)}{2}\right)^2}\right),
\end{equation}
with $\epsilon_2$ sufficiently small to be chosen.
Then the first $N_1+1$ agents reach distance $1$
to the other $N_2$ agents at the same time
$t$ such that $\epsilon(t)=\epsilon_2$.
We deduce that along the solution $x^2$ 
all agents converges to:
\[
\left(\frac{N_1}{N_1+1}+\frac{N_2}{N_2+N_1+1}
\epsilon_2\frac{1-N_1}{2},
\frac{N_2}{N_2+N_1+1}
\sqrt{1-\left(\frac{\epsilon_2(N_1+1)}{2}\right)^2}\right)
\]
which, for $N_2$ sufficiently large,
is close to $(y_1,y_2)$ of \eqref{eq:in-N2}.
Therefore the final distance between the asymptotic state
of the first agent along $x^1$ and $x^2$ is 
close to the norm of  \eqref{eq:in-N2}, which is
close to $\sqrt{2}$ for $N_1<N_2$ sufficiently big
and $\epsilon_2$ sufficiently small.\\
Now, by recursion, we can add $N_3>N_2$ agents
in position $(y_1+1 -\epsilon_{3,1},y_2-\epsilon_{3,2})$
choosing $\epsilon_{3,1}$ and $\epsilon_{3,2}$
so that all $1+N_1+N_2$ first
agents will reach distance $1$ at the same time to 
the other $N_3$ agents.
Reasoning as before, we can choose $N_3$ sufficiently
big and $\epsilon_{3,1}$, $\epsilon_{3,2}$,
sufficiently small so that the final distance between the asymptotic state of the first agent along $x^1$ and $x^2$ is 
close to $\sqrt{3}$.\\
By recursion, at each step we add a new group of agents at distance close to $1$ to the previous group along alternating
directions $(1,0)$ and $(0,1)$. In this way, 
for every $\nu$ we can choose  $1+N_1+\ldots+N_\nu$
agents in initial positions so that 
the final distance between the asymptotic state of the first agent along $x^1$ and $x^2$ is arbitrarily close to $\sqrt{\nu}$. Taking $\nu>R^2$ we conclude.
\end{proof}

\subsection{Non-uniqueness of CLSS solutions} \label{s-nonCLSS}

In this section, we show that CLSS solutions
may not be unique in dimension 2. We first study the easier case of the variant model \eqref{e:HK1}.
\begin{prop} \label{p-nonCLSS}
Consider the ODE \eqref{e:HK1} with $n=2$, $N=2$, $\phi_{ij}=1$
and initial condition $x_1(0)=(0,0)$, $x_2(0)=(1,0)$ and 
$x_3(0)=(\frac{1}{3},1)$. Then there exist two CLSS solutions
to the associated Cauchy problem.

The ODE \eqref{eq:HK} with the same initial data has a single CLSS solution.
\end{prop}
\begin{proof}
Consider an approximate solution $x(\cdot)$ having constant derivative
on the intervals with endpoints $0=t_0<t_1<\cdots <t_m=T$, $T$ sufficiently big, as in the Definition of CLSS solution (Definition \ref{def:ODE-sol}, case 5.)
If $x(t_i)=\frac{1}{3}$ for some $i$, then at time $t_i$ the first
agent is influenced by the third agent and the solution will tend to a unique
cluster. If $x(t_i)\not= \frac{1}{3}$ for every $i$, then the first agent
will never interact with the third agent, so the first two agents will
converge to $(\frac{1}{2},0)$ while the third will remain constant.
Since both situations can occur with arbitrarily close
times $t_i$, there are two CLSS solutions.

It is easy to prove that the ODE \eqref{eq:HK} with the same initial data has a single CLSS solution, as the case $x(t_i)=\frac13$ does not change the dynamics.
\end{proof}

We now prove non-uniqueness of a CLSS solution for the dynamics \eqref{eq:HK}. The construction is more complicated, as we study a system of 10 agents in $\R^2$ for which two Caratheodory solutions exists. We then give two sequences of approximated solutions, each uniformly converging to one of the two Caratheodory solutions. Detailed computations are omitted, as they can be numerically checked.

We first define the two Caratheodory solutions. We start by fixing the following constant values:

$$\eps:=1/10,\qquad T:=1/100,\qquad B=\frac{127}{10\sqrt{91}}.$$
We also fix the initial data for the system:
$$x_1(0)=(0,B-(B-\sqrt{1-(1/2-2\eps)^2})e^{2T}),\qquad x_{2}(0)=(0,B+(B-\sqrt{1-(1/2-2\eps)^2})e^{2T}),$$
$$x_{3,4,5,6}(0)=((1/2-2\eps)e^{8T},0),\qquad x_{7,8,9,10}(0)=(-(1/2-2\eps)e^{8T},0).$$

We have two Caratheodory solutions in which the groups of agents $\{3,4,5,6\}$ and $\{7,8,9,10\}$ are kept together. The first solution $x(t)$ is given by keeping agents $1,2$ not interacting with agents $3,\ldots, 10$. Thus, agents 1,2 exponentially converge to their barycenter $(0,B)$, following the trajectories
\begin{equation}\label{e-tr11}
x_1(t)=(0,B-(B-\sqrt{1-(1/2-2\eps)^2})e^{2T-2t}),\qquad x_2(t)=(0,B+(B-\sqrt{1-(1/2-2\eps)^2})e^{2T-2t}).
\end{equation}
The other agents exponentially converge to their barycenter $(0,0)$ following the trajectories
\begin{equation}\label{e-tr12}
x_{3,4,5,6}(t)=((1/2-2\eps)e^{8T-8t},0),\qquad x_{7,8,9,10}(t)=(-(1/2-2\eps)e^{8T-8t},0).
\end{equation}
A direct computation shows that, for all $t\in(0,+\infty)$ it holds $\|x_1(t)-x_2(t)\|<1$, $\|x_3(t)-x_7(t)\|<1$ and $\|x_2(t)-x_3(t)\|>\|x_1(t)-x_3(t)\|\geq 1$, with $\|x_1(t)-x_3(t)\|=1$ for $t=T$ only. This already shows that \eqref{e-tr11}-\eqref{e-tr12} is a classical and Caratheodory solution for \eqref{eq:HK}.

We now define a second Caratheodory solution for \eqref{eq:HK}  denoted by $y_i$ as follows:
\begin{equation}
y_i(t)=\begin{cases}x_i(t)\mbox{~~~for }i=1,\ldots,10\mbox{~and~}t\in[0,T], \\
\tilde Y^b_i(t-T)\mbox{~~~for }i=1,2\mbox{~and~}t\in[T,T_b],\\
\tilde Y^b_3(t-T)\mbox{~~~for }i=3,4,5,6\mbox{~and~}t\in[T,T_b],\\
\tilde Y^b_7(t-T)\mbox{~~~for }i=7,8,9,10\mbox{~and~}t\in[T,T_b],\\
\tilde Y^c_i(t-T_b)\mbox{~~~for }i=1,2\mbox{~and~}t\in[T_b,+\infty),\\
\tilde Y^c_3(t-T_b)\mbox{~~~for }i=3,4,5,6\mbox{~and~}t\in[T_b,+\infty),\\
\tilde Y^c_7(t-T_b)\mbox{~~~for }i=7,8,9,10\mbox{~and~}t\in[T_b,+\infty),\\
\end{cases}
\end{equation}
where $Y^b,Y^c$ are the unique solutions of the 8D-linear systems 
\begin{equation} \label{e-lineari}
 \dot Y^b=\left(\begin{array}{cccc}
-9\ \Id_2 & \Zero_2 & 4\ \Id_2 & 4\ \Id_2 \\
\Id_2 & -\Id_2 & \Zero_2 & \Zero_2 \\
\Id_2 & \Zero_2 & -5\ \Id_2 & 4\ \Id_2 \\
\Id_2 & \Zero_2 & 4\ \Id_2 & -5\ \Id_2 
\end{array}\right) Y^b,\qquad
 \dot Y^c=\left(\begin{array}{cccc}
-9\ \Id_2 & \Id_2 & 4\ \Id_2 & 4\ \Id_2 \\
\Id_2 & -9\ \Id_2 & 4\Id_2 & 4\Id_2 \\
\Id_2 & \Id_2 & -6\ \Id_2 & 4\ \Id_2 \\
\Id_2 & \Id_2 & 4\ \Id_2 & -6\ \Id_2 
\end{array}\right) Y^c
\end{equation}
starting at $(x_1,x_2,x_3,x_7)(T)$ and $(x_1,x_2,x_3,x_7)(T_b)$, respectively, where $T_b$ is the first time for which $\|x_2(t)-x_3(t)\|=1$. Here, the notations $\Id_2,\Zero_2$ denote the identity and zero matrices of dimension 2, respectively. The definition of matrices and time $T_b$ reflect the fact that $y$ represents the case in which agent $1$ starts interacting with agents $3,\ldots,10$ at time $T$. This ensures that agents move closer, up to time $T_b$ in which agent $2$ also starts interacting with all agents. It is easy to prove that $y(t)$ is a Caratheodory solution to \eqref{eq:HK}. Finally, 
one can prove that there exist no other Caratheodory solutions to \eqref{eq:HK} keeping together  the agents in the groups 3,4,5,6 and 7,8,9,10. 

We now build two sequence of sample-and-hold solutions for \eqref{eq:HK}. First fix the finite time interval $[0,\bar{T}]$ with $\bar{T}\in(T,T_b)$ chosen to be of the form $\frac{r+1}{r} T$ with $r\in \N\setminus \{0\}$. This allows us to focus on the simple case in which agent 2 does not interact with agents $3,\ldots,10$. Define the parameter $\Delta t:=\bar{T}/K$ for some $K\in  r^2(\N\setminus\{0,1\})^2$, i.e. $K$ being a positive multiple of $r^2$ and a perfect square strictly larger than $3r^2$. This choice ensures that $K-\sqrt{K},K+3\sqrt{K},\frac{r+1}{r}K\in\N$ and $K+3\sqrt{K} \leq \frac{r+1}{r}K$; these properties will be useful in the following.

Consider now the sample-and-hold solution defined on the uniform sequence $t_k:=k\Delta t$ as 
\begin{equation}\label{e-Euler}
y^K_i(t_{k+1})=y_i^K(t_k)+\sum_{j\neq i} a_{ij}(\|y^K_i(t_k)-y^K_j(t_k)\|) (y^K_j(t_k)-y^K_i(t_k)),
\end{equation}
starting from $y^K_i(0)=x_i(0)$. Direct computations show that $y^K(t)$ converges to $y(t)$, i.e.  to the second Caratheodory solution given above. Indeed, it is sufficient to check that, if $y^K(t)$ converges to $x(t)$ on $[0,T]$, then it holds $$\|y^K_1(T)-y^K_3(T)\|^2=1-\frac{27}{62500}K^{-1}+o(K^{-1})<1$$ for $K$ sufficiently large, i.e. agent 1 starts interacting with agents $3,\ldots,10$. This implies that $y^K(t)$ starts converging to the solution $y(t)$ on $[T,\bar{T}]$, as it is the only Caratheodory solution keeping groups 3,4,5,6 and 7,8,9,10 coinciding and agent 1 interacting with them.

We now build a second sequence of sample-and-hold solutions, now converging to the Caratheodory solution $x(t)$. We use the standard sample-and-hold solution with uniform step $\Delta t=T/K$ defined above, up to time $\Delta t (K-\sqrt{K}).$ We then use a single time step of length $4\sqrt{K}\Delta t$, then the necessary steps $\frac{\bar T-T-3\sqrt{K}}{\Delta t}$ of length $\Delta t$ to reach time $\bar T$. We denote such solution by $x^K(t)$. Direct computations show that $\|x^K_1(t)-x^K_3(t)\|$ is decreasing for $t< (K-\sqrt{K})\Delta t$ and increasing for $t> (K+3\sqrt{K})\Delta t$. Moreover, it holds 
$$\|x^K_1((K-\sqrt{K})\Delta t)-x^K_3((K-\sqrt{K})\Delta t)\|=1 + \frac{36}{56875} K^{-1}+o(K^{-1}),$$
$$\|x^K_1((K+3\sqrt{K})\Delta t)-x^K_3((K+3\sqrt{K})\Delta t)\|=1 + \frac{3186}{1421875} K^{-1}+o(K^{-1}).$$
This implies that, for $K$ sufficiently large, the sequence $x^K$ defined above satisfies  $\|x^K_1(t_k)-x^K_3(t_k)\|> 1$ for each $t_k$. As a consequence, for each $K$, there is no interaction between agent 1 and agents $3,\ldots,10$. The sequence $x^K$ then converges to the Caratheodory solution $x(t)$.

\section{Uniqueness results}\label{sec:uniq}
In this section, we provide positive results
for uniqueness. As shown in Sections
\ref{sec:HK-R} and \ref{sec:HK-multiR}, uniqueness
fails for most concept of solutions. However,
this can be guaranteed for almost every initial data
for Filippov, thus also for Krasovskii and
Caratheodory solutions.

Recall Definition \ref{def:M-disc} and 
observe that each $\MM_{ij}$ is a (smooth) manifold of codimension 1 in $\R^{nN}$, i.e.  $\mathrm{dim}(\MM_{ij})=nN-1$. This implies that $\MM$ is a stratified set of codimension 1.

We first need the following auxiliary result about uniqueness of Filippov solutions. It shows that uniqueness can be lost only after reaching $\MM$.
\begin{prop} \label{prop:uniq-M} 
Let $x(\cdot),y(\cdot)$ be Filippov solutions to \eqref{eq:HK} defined 
on the time interval $[0,T]$, with $T>0$, that satisfy $x(t),y(t)\not\in \MM$ for all $t\in [0,T)$ and $x(T)=y(T)$. It then holds $x(t)=y(t)$ for all $t\in [0,T]$.\\
Similarly if $x(t),y(t)\not\in \MM$ for all $t\in (0,T]$ and $x(0)=y(0)$, then it holds $x(t)=y(t)$ for all $t\in [0,T]$.
\end{prop}
\begin{proof} 
Recall that \eqref{eq:HK} is uniformly Lipschitz continuous
on the open set $\R^{nM}\setminus\MM$.
For the first statement, since $x(t),y(t)\not\in \MM$ on $[0,T)$, then 
the functions $x(\cdot)$ and $y(\cdot)$  are differentiable and satisfy \eqref{eq:HK} on $[0,T)$.
For every $\epsilon$ we can apply Gronwall Lemma backward in time on $[0,T-\epsilon[$, thus getting 
\[
\|x(t)-y(t)\|\leq e^{L(T-\epsilon)}\|x(T-\epsilon)-y(T-\epsilon)\|.
\]
By letting $\epsilon\to 0$ we conclude.\\
The second statement follows similarly, by using Gronwall Lemma forward in time.
\end{proof}

We then prove the following result.
\begin{prop} \label{p-aeuniq}
Consider the system \eqref{eq:HK} with $\phi_{ij}$ Lipschitz continuous.
The set of initial data $\bar x\in \R^{nN}$ for which there exist more than one Filippov solutions for \eqref{eq:HK} has zero Lebesgue measure in $\R^{nN}$.
\end{prop}

\begin{proof}
Given an initial condition $\bar{x}$, we define $X_{\bar{x}}$ 
to be the set of solutions  $x(\cdot)$ to \eqref{eq:HK} defined  
on some time interval  $[0,T(x(\cdot))[$, with $0<T(x(\cdot))\leq +\infty$, and 
satisfying $x(0)=\bar{x}$. We set:
\begin{equation}\label{eq:def-tU}
t_U=\inf \{t:  \exists x(\cdot),y(\cdot)\in X_{\bar{x}}, 
t\leq\min \{T(x(\cdot)),T(y(\cdot))\}, x(t)\not= y(t)\},
\end{equation}
thus lack of uniqueness occurs when $t_U<+\infty$.
Since $\MM$ is a stratified set of codimension 1, it has zero Lebesgue measure in  $\R^{nN}$. Thus we only need to prove that the following set
has zero Lebesgue measure:
\begin{equation}\label{eq:def-A}
A=\{\bar{x}\in \R^{nN}\setminus\MM: t_U<+\infty\}.
\end{equation}
For $\bar x\in A$, we define:
\begin{equation}
\tilde{t}=\inf \{t: \exists x(\cdot)\in X_{\bar{x}}, x(t)\in\MM\}.
\end{equation}
Since \eqref{eq:HK} is Lipschitz continuous on $ \R^{nN}\setminus\MM$,
all solutions $x(\cdot)$ in $X_{\bar{x}}$ coincide up to time $\tilde{t}$, and,
since $\MM$ is closed, we have $x(\tilde{t})\in\MM$.
Therefore $\tilde{x}=x(\tilde{t})$, with
$x(\cdot)\in X_{\bar{x}}$, depends only on $\bar{x}$
and not on the chosen solution $x(\cdot)\in X_{\bar{x}}$.
Now, given $i,j\in\{1,\ldots,N\}$, $i\not=  j$, consider the
expression:
\begin{equation}\label{eq:def-alphaij}
\alpha_{ij}(x)=(x_i-x_j)\cdot 
\left(\sum_{k\not= i,j}  a_{i,k}(\|x_i-x_k\|)  (x_k-x_i)-
\sum_{k\not= i,j}  a_{j,k}(\|x_j-x_k\|)  (x_k-x_j)
\right)
\end{equation}
and define the following sets:
\begin{equation}
\widehat{\MM}_{ij}=\{x\in\MM_{ij}: \alpha_{ij}(x)\in\{0,2\}\}.
\end{equation}
Define the set of quadruplets of indexes $I \subset \{1,\ldots,N\}^4$ by
$$I=\{(i,j,k,l): i\neq j, k\neq l, (i,j)\neq (k,l),  (i,j)\neq (l,k)\}.$$
Finally set:
\begin{equation}
\widehat{\MM}_{ijkl}= \MM_{ij}\cap \MM_{kl},\quad
\widehat{\MM} = \cup_{i,j:i\not= j} \widehat{\MM}_{ij} \bigcup
\cup_{(i,j,k,l)\in I} \widehat{\MM}_{ijkl}.
\end{equation}
Notice that each set $\widehat{\MM}_{ij}$ is of codimension two and the same is true for $ \widehat{\MM}_{ijkl}$
if $(i,j,k,l)\in I$. Therefore $\widehat{\MM}$ is
of codimension 2.
We now state the following claim:\\
{\bf Claim a)} If $\tilde{x}\in M\setminus\widehat{\MM}$, then
there exists  $\epsilon>0$ such that  $x(t)\notin\MM$
for $t\in ]\tilde{t},\tilde{t}+\epsilon[$, and
$x\equiv y$ on $[0,\tilde{t}+\epsilon[$
for every $x(\cdot)$, $y(\cdot) \in X_{\bar{x}}$.\\

We  prove Claim a). By assumption there exists a unique (not ordered) couple $\{i,j\}$, $i\not= j$, such that $\tilde{x}\in\MM_{ij}$. Define the function:
\begin{equation}\label{eq:def-thetaij}
\theta_{ij}(t) = \|x_i(t)-x_j(t)\|^2.
\end{equation}
Notice that $\theta_{ij}$ is twice continuously differentiable on $[0,\tilde{t}[$
with bounded derivatives, thus we can define
$\tilde{\xi}=\lim_{t\to\tilde{t}-}\dot{\theta}_{ij}(t)$, i.e. the left limit
of the first derivative of $\theta_{ij}$ at $\tilde{t}$.\\
Assume first $\tilde{\xi}>0$. Then there exists $\epsilon>0$
such that $\theta_{ij}$ is strictly increasing
on $]\tilde{t}-\epsilon,\tilde{t}[$,
thus $\theta_{ij}(t)<1$ on $]\tilde{t}-\epsilon,\tilde{t}[$.
Moreover, since $x(\tilde{t})\notin\widehat{\MM}$, possibly restricting
$\epsilon$, we have that $x(t)\notin \widehat{\MM}$ on  
$]\tilde{t}-\epsilon,\tilde{t}+\epsilon[$. 
Recalling \eqref{eq:def-alphaij}, we deduce $\tilde{\xi}=\alpha_{ij}(\tilde{x})-2>0$. Now, given $x(\cdot)\in  X_{\bar{x}}$ 
for almost every $t\in ]\tilde{t},\tilde{t}+\epsilon[$
there exists $\beta_1$, $\beta_2\in [0,1]$ such that:
\begin{eqnarray*}
\dot{\theta}_{ij}(t)&=& (x_i(t)-x_j(t))\cdot\\
&&
\left(\sum_{k\not= i,j}  a_{i,k}  (x_k(t)-x_i(t))-
\sum_{k\not= i,j}  a_{j,k} (x_k(t)-x_j(t))
+\beta_1 (x_j(t)-x_i(t))-\beta_2 (x_j(t)-x_i(t))\right),
\end{eqnarray*}
where we omitted the arguments of $a_{i,k}$ and $a_{j,k}$ for simplicity.
It follows $\dot{\theta}_{ij}(t)\geq \alpha_{ij}(x(t))-2>0$ for $\epsilon$
sufficiently small. Thus $\theta_{ij}(t)>1$ on $]\tilde{t},\tilde{t}+\epsilon[$
and $x(t)\notin\MM$ on the same interval. 
Proposition \ref{prop:uniq-M} implies that
all solutions coincide on the interval $]\tilde{t},\tilde{t}+\epsilon[$
and we proved Claim a) for $\tilde{\xi}>0$.\\
The case $\tilde{\xi}<0$ can be treated in an entirely similar way,
concluding that $\theta_{ij}<1$ on on $]\tilde{t},\tilde{t}+\epsilon[$.
Finally, notice that the case $\tilde{\xi}=0$ is excluded since
$\tilde{x}\notin\widehat{\MM}$.
The proof of Claim a) is finished.

We now state the next claim:

{\bf Claim b)} Assume there exists $x(\cdot)\in X_{\bar{x}}$ such
that $x(t)\notin\widehat{\MM}$ for $t\in [0,T[$, $T>0$. Then, it either holds $x(T)\in \widehat{\MM}$ or there exists $\eps>0$ such that all solutions in $X_{\bar{x}}$ coincide in $[0,T+\eps)$.\\

We now prove Claim b). Given $\bar x$, we find $\tilde t$ such that $x(t)\not in \MM$ for all $t\in[0,\tilde t)$ and $\tilde{x}=x(\tilde t)\in \MM$. If $\tilde x\in \widehat{\MM}$, then $T=\tilde t$. Otherwise, define $t^1_{\MM}=\tilde{t}$ and, using Claim a), extend $x(\cdot)$ on the time interval $(t^1_{\MM},t^2_{\MM})$, where the right extremum is given by
\begin{equation}
t^2_{\MM}=\inf \{t: t>\tilde{t}, \exists x(\cdot)\in X_{\bar{x}}, x(t)\in\MM\}.
\end{equation}
If $x(t^2_\MM)\notin\widehat{\MM}$, then we can define
$t^3_\MM$ and so on. That is, as long as the trajectory from $\bar{x}$
does not reach $\widehat{\MM}$, we can set:
\begin{equation}\label{eq:def-tiMM}
t^{\nu}_{\MM}=\inf \{t: t>t^{\nu-1}_{\MM}, \exists x(\cdot)\in X_{\bar{x}}, x(t)\in\MM\}.
\end{equation}
Again by Claim a), observe that the trajectory starting from $\bar x$ is unique on all intervals $[0,t^\nu_{\MM}]$. Thus, if there exists $t^\nu_{\MM}>T$, there exists $\eps>0$ such that all solutions in $X_{\bar{x}}$ coincide in $[0,T+\eps)$.

Otherwise, assume that $t^\nu_{\MM}\leq T$ for all $\nu$. Since this is an increasing and bounded function, it admits a limit $\bar t:=\lim_{\nu\to+\infty} t^\nu_{\MM}$. We aim to prove $\bar t=T$. Since pairs $i,j$ are in finite number, there exists $i,j$, $i\not= j$, anda subsequence, still indicated by $t^\nu_{\MM}$, such that $x(t^\nu_{\MM})\in\MM_{ij}$ and thus $\theta_{ij}(t^\nu_{\MM})=1$ (see \eqref{eq:def-thetaij}). 
From the proof of Claim a), we deduce that $\theta_{ij}$ is
continuously differentiable and not equal to $1$
on every interval $]t^\nu_{\MM},t^{\nu+1}_{\MM}[$,
and satisfies $\theta_{ij}(t^\nu_{\MM})=\theta_{ij}(t^{\nu+1}_{\MM})=1$.
Thus there exists $\tau_\nu\in ]t^\nu_{\MM},t^{\nu+1}_{\MM}[$
such that $\dot{\theta}_{ij}(\tau_\nu)=0$.
From the proof of Claim a),
we have that for every $\nu$ either
$\dot{\theta}_{ij}(\tau_\nu)=\alpha_{ij}(x(\tau_\nu))$
or $\dot{\theta}_{ij}(\tau_\nu)=\alpha_{ij}(x(\tau_\nu))-2$.
Passing to the limit in $\nu$, we get $x(\bar{t})\in\widehat{\MM}$. We thus have $\bar{t}=T$, that proves the claim.\\

We now use Claim b) to prove the main result. Define:
\begin{equation}
t_{\widehat{\MM}}=
\inf \{t: \exists x(\cdot)\in X_{\bar{x}}, x(t)\in\widehat{\MM)}\},
\end{equation}
and recall the definition of $t_U$ in \eqref{eq:def-tU}. Claim b) ensures that, for $t_U<+\infty$, it holds$t_{\widehat{\MM}}\leq t_U$.

This implies that 
$A=\{\bar{x}\in \R^{nN}\setminus\MM: t_{\widehat{\MM}}\leq t_U<+\infty\}$.
We denote by $H^r$ the Hausdorff measure of dimension $r$ in $\R^{nN}$. Each $x(\cdot)\in X_{\bar{x}}$, $\bar{x}\in A$,
is Lipschitz continuous and, by Proposition \ref{prop:uniq-M},
it coincides (at least) up to $t_{\widehat{\MM}}$, 
thus $H^{1+\epsilon}(\{x(t):t\in [0,t_{\widehat{\MM}}],x(\cdot)\in X_{\bar{x}}\})=0$ for every $\epsilon>0$. 
By Fubini Theorem, since $\widehat{\MM}$ is of codimension 2,
for $0<\epsilon<1$ we have:
\[
H^{nN}(A)\leq \int_{\widehat{\MM}} 
\left(H^{1+\epsilon}(\{x(t):t\in [0,t_{\widehat{\MM}}],x(\cdot)\in X_{\bar{x}}\})
\right)
d H^{nN-2+\epsilon}(\bar{x})=0.
\]
Since $H^{nN}$ coincides with the Lebesgue measure on $\R^{nN}$, the set $A$ has zero measure.
\end{proof}

\section{Property P2): clustering for general solutions}\label{sec:P2}
In this section, we discuss property P2), also called clustering, for solutions of \eqref{eq:HK}. We prove that P2) holds for Filippov solutions, even though the limit is not uniquely determined by the initial data. This implies that  P3) fails, as already shown in Section \ref{sec:HK-R}.

First observe that \eqref{eq:HK} can be written as a gradient flow as follows. Define 
$$\Phi_{ij}(r)=\begin{cases}
\int_0^r\phi_{ij}(s)s\,ds&\mbox{~~for~}r<1\\
\int_0^1\phi_{ij}(s)s\,ds&\mbox{~~for~}r\geq 1
\end{cases}
$$
and observe that, if $\|x_i-x_j\|\neq 1$, for every $i\not= j$, then
$$\dot x_i=-\sum_{j\neq i} \nabla \Phi_{ij}(|x_i-x_j|).$$
This suggests to define the following candidate Lyapunov function:
$$V(x)=\sum_{i,j\neq i} \Phi_{ij}(|x_i-x_j|)$$ 
and observe that it holds $\dot V(x(t))\leq 0$ for a.e. time, \FR{since $\nabla V(x)\cdot v\leq 0$ for each $v\in F(x)$.}  We now prove the following result about clustering.
\begin{prop} \label{p-Fclust} Let $x_1(t),\ldots,x_N(t)$ be a Filippov  solution of \eqref{eq:HK}. The following clustering properties hold:
\begin{itemize}
\item each agent satisfies $\lim_{t\to+\infty} x_i(t)=x_i^\infty$ for some $x_i^\infty \in \R^n$;
\item the limits satisfy the following: for each $i\neq j$ it either holds $x_i^\infty=x_j^\infty$ or $\|x_i^\infty-x_i^\infty\|\geq 1$.
\end{itemize}

The same result holds for the variant model \eqref{e:HK1}.
\end{prop}
\begin{remark} One might try to use a LaSalle principle to prove this result. Even though the proof below is based on the same ideas, we need to observe that $V$ is not proper, it is not differentiable, and that the largest invariant set of $\nabla V=0$ (for any reasonable definition of it) is never reduced to a point.
\end{remark}
\begin{proof} The proof is identical in the two cases  \eqref{eq:HK}-\eqref{e:HK1}. It is first necessary to observe that $V$ is not proper, i.e. it does not satisfy $V(x)\to +\infty$ when $|x|\to+\infty$, as it depends on pairwise distances only. Nevertheless, recall that the set $\Omega(t):=\mathrm{co}(x_i(t))$ is weakly contracting, see Proposition \ref{p-contractive}. As a consequence, we have that $x(t)$ is a compact trajectory, hence it converges to its $\omega$-limit, that is bounded.

Let now $x^\infty=(x_1^\infty,\ldots,x_N^\infty)$ being a point in the $\omega$-limit and assume that it exists $i,j$ such that $|x_i^\infty-x_j^\infty|=L\in(0,1)$. By definition of $\omega$-limit, it exists an increasing sequence $t^k\to+\infty$ such that $|x_i(t^k)-x_j(t^k)|\in(L-\eps,L+\eps)$ for any $\eps>0$. By observing that velocities for \eqref{eq:HK} are bounded, there exists a uniform $\delta$ such that $|x_i(t)-x_j(t)|\in(L-2\eps,L+2\eps)$ for all $t\in (t^k-\delta,t^k+\delta)$. Choose $\eps<\frac12 \min(L,1-L)$ and observe that 
$$V(x(t^k+\delta))-V(x(t^k-\delta))\leq -2\delta \phi_{ij}(L-2\eps).$$ By eventually taking a subsequence of $t^k$, one can always assume $t^k+\delta<t^{k+1}-\delta$. By recalling that $V$ is decreasing in time, it then holds $V(t^k+\delta)\leq V(x(0))-2\delta \phi_{ij}(L-2\eps) k$, hence $\lim_{k\to+\infty} V(t^k+\delta)=-\infty$. This contradicts the fact that $V$ is bounded from below.

We have now proved that any $x^\infty$ in the $\omega$-limit satisfies either $x_i^\infty=x_j^\infty$ or $\|x_i^\infty-x_j^\infty\|\geq 1$. We now need to prove that the $\omega$-limit is reduced to a point. We first define the transitive relation 
$$i\sim j\mbox{~~ when~~}\lim_{t\to+\infty} x_i(t)-x_j(t)=0.$$

It is crucial to observe that either it holds $i\sim j$ or $\lim\inf_{t\to+\infty}|x_i(t)-x_j(t)|\geq 1$. Indeed, if $\lim\inf_{t\to+\infty}|x_i(t)-x_j(t)|\in(0,1)$, there exist times $t^k\to+\infty$ such that $|x_i(t^k)-x_j(t^k)|\in (L-\eps,L+\eps)$ with $L\in(0,1)$ and $\eps>0$ sufficiently small, that in turn ensure $\lim_{k\to+\infty} V(t^k+\delta)=-\infty$ as explained above. Contradiction.

It then makes sense to define clusters $C^1,\ldots,C^k$, each being the class of equivalence of indexes $i=1,\ldots N$ with respect to $\sim$. By definition, given $\eps>0$,  there exists a time $T_0$ such that all agents satisfy either $|x_i(t)-x_j(t)|<\eps$ or $|x_i(t)-x_j(t)|>1-\eps$ for all times $t>T_0$. Eventually translating such time, we assume $T_0=0$ from now on.

We are now ready to prove that the center of each cluster converges. Let $y^1$ be the center of cluster $C^1$, i.e. $y_1=\frac{1}{N_1}\sum x_i$ where $N_1$ is the number of elements of $C^1$. Since $\dot y^1$ is defined for almost every time, \FR{one can write $\dot y^1$ by the following observation: for each index $i\in C^1$, the contribution to $\dot x_i$ of each agent $j$ satisfying $\|x_i-x_j\|\neq 1$ is uniquely determined, while for $j$ such that $\|x_i-x_j\|=1$ there exists $\alpha_{ij}\in[0,1]$ such that the contribution is $\alpha_{ij}a_{ij}(x_j-x_i)$. By choosing $\alpha_{ij}=1$ for $j\not \in C^1$ with $\|x_i-x_j\|\neq 1$, one has}

\begin{eqnarray*}
\int_0^T |\dot y^1(t)|\,dt&=&\frac{1}{N_1}\int_0^T\left| \sum_{i,j\in C^1} a_{ij}(x_j-x_i)+\sum_{i\in C^1,j\not\in C^1} \alpha_{ij}(t)a_{ij}(x_j-x_i)\right|\leq\\
&\leq& 0+\frac{1}{N_1}\int_0^T\sum_{i\in C^1,j\not\in C^1} a_{ij} |x_i(t)-x_j(t)|\leq 
\frac{1}{N_1(1-\eps)}\int_0^T \sum_{i,j\neq i}a_{ij}|x_i(t)-x_j(t)|^2=\\
&=&\frac{1}{N_1(1-\eps)}\int_0^T-\dot V(x(t))\,dt\leq \frac{V(0)-V(T)}{N_1(1-\eps)}<+\infty.
\end{eqnarray*}
where we first used antisymmetry for $i,j\in C^1$, then we recalled that $V(t)$ is decreasing and bounded from below. Since $\dot y_1$ is integrable, then $y^1(t)$ admits a limit for $t\to +\infty$. Since $x_i-x_j\to 0$ for all $i,j\in C^1$ and the center of the cluster admits a limit, then all $x_i$ converge to such limit. \end{proof}

\begin{remark} For the system  \eqref{e:HK1} final clusters may be
at distance one as shown by next example with three agents in dimension two.
Consider the initial condition
$x_{1,0}=(0,\eps)$, $x_{2,0}=(0,-\eps)$, and $x_{3,0}=(1,0)$,
with $\eps< \frac12$. The unique Caratheodory solution is given by
$$x_{1,0}=(0,\eps e^{-t}),\qquad x_{2,0}=(0,-\eps e^{-t}),
\qquad x_{3,0}=(1,0),$$
thus converging to two clusters $(0,0)$ and $(1,0)$.
\end{remark}}

\section{Proof of main Theorems}\label{s-main}
In this section, we prove the three theorems stated in the introduction. Proofs actually collect results given in previous sections.\\

{\noindent{\textsc{Proof of Theorem \ref{t-exun}.}}} Existence of solutions in the Filippov, Krasovskii, Caratheodory, CLSS and stratified sense was shown in Section \ref{s-existence}. Non-existence of classical solutions is well-known, as shown in the counterexample in Proposition \ref{p-nonex-class}.\\
Uniqueness of classical solutions is standard, using Cauchy-Lipschitz argument of uniqueness.
Non-uniqueness of Filippov, Krasovskii, and Caratheodory solutions is proved by the counterexamples of Proposition \ref{prop:toy-ex}. Non-uniqueness of CLSS solutions is proved in Proposition \ref{p-nonCLSS}. 
Uniqueness of stratified solutions for a fixed stratification is given by definition.\\
Uniqueness of Filippov solutions for almost every initial data was proved in Proposition \ref{p-aeuniq}. This induces the same result for all other concepts of solutions, due to Proposition \ref{p-inclusion}. {$\hfill\Box$\vspace{0.1 cm}\\}

{\noindent{\textsc{Proof of Theorem \ref{t-prop}.}}} Filippov and Krasovskii solutions coincide, and they do not satisfy P3) as shown in Proposition \ref{prop:bar-clu-2a}. They satisfy P1), as proved in Proposition \ref{prop:FK-HK}. They satisfy P2), as shown in Proposition \ref{p-Fclust}.\\
Classical, Caratheodory, CLSS, stratified solutions satisfy P1)-P2) due to the inclusion in Filippov solutions, see Proposition \ref{p-inclusion}.

Caratheodory solutions do not satisfy P3), again by Proposition \ref{prop:bar-clu-2a}. CLSS solutions do not satisfy P3), as shown by the counterexample in Section \ref{s-nonCLSS}.

Classical solutions satisfy property P3), as a direct consequence of uniqueness of the solution. Similarly, stratified solutions for a fixed stratification are unique, by definition, thus satisfy P3). However, the asymptotic
state depends on the stratification as shown in Proposition \ref{prop:toy-ex}.
{$\hfill\Box$\vspace{0.1 cm}\\}

{\noindent{\textsc{Proof of Theorem \ref{t-prop2}.}}} Krasovskii and Filippov multifunctions are insensitive to the value of $a_{ij}(1)$ by definition, thus the structure of Krasovskii and Filippov solutions does not change.\\
As for classical solutions, consider the Cauchy problem \eqref{eq:HK1d2a}-\eqref{eq:ic-toy}, as in Proposition \ref{prop:toy-ex}. If $a_{ij}(1)=0$, the Proposition states that the only classical solution is the constant one. Instead, if $a_{ij}(1)=1$, it is easy to prove that the unique classical solution is given by $$x_1(t)-x_2(t)=e^{-2t} (x_{1,0}-x_{2,0}),\qquad
x_1(t)+x_2(t)=x_{1,0}+x_{2,0}.$$
Caratheodory solutions of the variant model \eqref{e:HK1} are different than Caratheodory solutions of \eqref{eq:HK}, as shown in the examples of Poposition \ref{p-Cara1incl}.\\
CLSS solutions are distinct in the two cases, as shown in Proposition \ref{p-nonCLSS}.\\
For stratified solution, consider the Cauchy problem \eqref{eq:HK1d2a}-\eqref{eq:ic-toy}, as in Proposition \ref{prop:toy-ex}. For \eqref{e:HK1} the first stratification $S_1$ is not admissible, thus stratified solutions
are different than those for \eqref{eq:HK}.\\
Proof and counterexamples for Properties P1-2-3) are identical to the study of \eqref{eq:HK}.

{$\hfill\Box$\vspace{0.1 cm}\\}


\bibliographystyle{abbrv}
\bibliography{biblio}
\end{document}